\documentclass[twoside,11pt]{article}

\usepackage{jmlr2e}

\usepackage[utf8]{inputenc}

\RequirePackage
{hyperref}

\usepackage{amsmath}
\usepackage{amsfonts}

\usepackage[shortlabels]{enumitem}

\usepackage{color}

\makeatletter
\newcommand{\labitem}[2]{%
\def\@itemlabel{\textbf{#1}}
\item
\def\@currentlabel{#1}\label{#2}}
\makeatother

\newcommand{\Acal}{\mathcal{A}}
\newcommand{\Bbf}{\mathbf{B}}
\newcommand{\Bcal}{\mathcal{B}}

\newcommand{\Ebb}{\mathbb{E}}

\newcommand{\Fcal}{\mathcal{F}}
\newcommand{\gbf}{\mathbf{g}}
\newcommand{\Gcal}{\mathcal{G}}
\newcommand{\Hbf}{\mathbf{H}}

\newcommand{\Kbf}{\mathbf{K}}
\newcommand{\Kcal}{\mathcal{K}}
\newcommand{\Lbf}{\mathbf{L}}

\newcommand{\Mcal}{\mathcal{M}}
\newcommand{\Nbb}{\mathbb{N}}

\newcommand{\Pbb}{\mathbb{P}}

\newcommand{\Qbf}{\mathbf{Q}}
\newcommand{\Qcal}{\mathcal{Q}}

\newcommand{\Rbb}{\mathbb{R}}

\newcommand{\Sbf}{\mathbf{S}}

\newcommand{\Tcal}{\mathcal{T}}

\newcommand{\Vbf}{\mathbf{V}}

\newcommand{\Xcal}{\mathcal{X}}
\newcommand{\Ycal}{\mathcal{Y}}
\newcommand{\Zbb}{\mathbb{Z}}

\newcommand{\Csigma}{C_\Qbf}
\newcommand{\CB}{C_\gamma}
\newcommand{\Caux}{{C_\text{aux}}}

\newcommand{\Gref}{{G_\lambda}}
\newcommand{\mix}{{\text{mix}}}
\newcommand{\one}{\mathbf{1}}
\newcommand{\pen}{\text{pen}}

\DeclareMathOperator{\argmax}{{arg max}}
\DeclareMathOperator{\Cov}{{Cov}}
\DeclareMathOperator{\Corr}{{Corr}}
\DeclareMathOperator{\Var}{{Var}}

\renewcommand{\leq}{\leqslant}
\renewcommand{\geq}{\geqslant}

\ShortHeadings{Oracle inequality for misspecified NPHMMs}{L. Lehéricy} 
\firstpageno{1}

\begin{document}

\title{Nonasymptotic control of the MLE for misspecified nonparametric hidden Markov models}

\author{\name Luc Lehéricy \email luc.lehericy@univ-cotedazur.fr \\
       \addr Laboratoire J. A. Dieudonn\'e \\
       Universit\'e C\^ote d'Azur, CNRS \\
       06108, Nice, France}

\maketitle

\begin{abstract}
Finite state space hidden Markov models are flexible tools to model phenomena with complex time dependencies: any process distribution can be approximated by a hidden Markov model with enough hidden states.
We consider the problem of estimating an unknown 
process distribution using nonparametric hidden Markov models in the \emph{misspecified setting}, that is when the data-generating 
process may not 
be a hidden Markov model.
We show that when the true distribution is exponentially mixing and satisfies a forgetting assumption, the maximum likelihood estimator recovers the best approximation of the true distribution. We prove a finite sample bound on the resulting error and show that it is optimal in the minimax sense\---up to logarithmic factors\---when the model is well specified.
\end{abstract}




\tableofcontents

\section{Introduction}

Let $(Y_1, \dots, Y_n)$ be a sample following some unknown distribution $\Pbb^*$.
The maximum likelihood estimator can be formalized as follows:
let $\{\Pbb_\theta\}_{\theta \in \Theta}$, the \emph{model}, be a family of possible distributions; 
pick a distribution $\Pbb_{\hat{\theta}}$ of the model which maximizes the likelihood of the observed sample.

In many situations, the true distribution may not belong to the model at hand: this is the so-called \emph{misspecified setting}. One would like the estimator to give sensible results even in this setting. This can be done by showing that the estimated distribution converges to the best approximation of the true distribution within the model.
The goal of this paper is to establish a finite sample bound on the error of the maximum likelihood estimator for a large class of true distributions and a large class of nonparametric hidden Markov models.

In this paper, we consider maximum likelihood estimators (shortened MLE) based on model selection among finite state space hidden Markov models (shortened HMM).
A finite state space hidden Markov model is a stochastic process $(X_t, Y_t)_t$ where only the observations $(Y_t)_t$ are observed, such that the process $(X_t)_t$ is a Markov chain taking values in a finite space and such that the $Y_s$ are independent conditionally to $(X_t)_t$ with a distribution depending only on the corresponding $X_s$. The parameters of a HMM $(X_t, Y_t)_t$ are the initial distribution and the transition matrix of $(X_t)_t$ and the distributions of $Y_s$ conditionally to $X_s$.

HMMs have been widely used in practice, for instance in climatology \citep{LWM03}, ecology \citep{boyd2014HmmUsesOfBirdTrajectories}, voice activity detection and speech recognition \citep{CC00, Lef03}, biology \citep{YPRH11, VBMMR14}... 
One of their advantages is their ability to account for complex dependencies between the observations: despite the seemingly simple structure of these models, the fact that the process $(X_t)_t$ is hidden makes the process $(Y_t)_t$ non-Markovian.

Up to now, most theoretical work in the literature focused on well-specified and parametric HMMs, where a smooth parametrization by a subset of $\Rbb^d$ is available, see for instance \cite{baum1966HMM} for discrete state and observations spaces, \cite{leroux92MLEHMM} for general observation spaces and \cite{douc2001asymptotics} and \cite{douc2011consistency} for general state and observation spaces.
Asymptotic properties for misspecified models have been studied recently by \cite{mevel2004misspecified} for consistency and asymptotic normality in finite state space HMMs and \cite{douc2012misspecified} for consistency in HMMs with general state space. Let us also mention \cite{pouzo2016misspecified}, who studied a generalization of hidden Markov models in a semi-misspecified setting. All these results focus on parametric models.

Few results are available on nonparametric HMMs, and all of them focus on the well-specified setting. \cite{AH14} prove consistency of a nonparametric maximum likelihood estimator based on finite state space hidden Markov models with nonparametric mixtures of parametric densities. \cite{vernet2015posterior, vernet2015posteriorrates} study the posterior consistency and concentration rates of a Bayesian nonparametric maximum likelihood estimator. Other methods have also been considered, such as spectral estimators in \cite{AHK12, HKZ12, dCGLLC15, bonhomme2016multiview, lehericy2017sbs} and least squares estimators in \cite{dCGL15, lehericy2017sbs}. Besides \cite{vernet2015posteriorrates}, to the best of our knowledge, there has been no result on convergence rates or finite sample error of the nonparametric maximum likelihood estimator, even in the well-specified setting.

The main result of this paper is an oracle inequality that holds as soon as the models have controlled tails. This bound is optimal when the true distribution is a HMM taking values in $\Rbb$. Let us give some details about this result.

Let us start with an overview of the assumptions on the true distribution $\Pbb^*$.
The first assumption is that the observed process is strongly mixing.
Strong mixing assumptions can be seen as a strengthened version of ergodicity. They have been widely used to extend results on independent observation to dependent processes, see for instance \cite{bradley2005strongmixingsurvey} and \cite{dedecker2007weakdependence} for a survey on strong mixing and weak dependence conditions.
The second assumption is that the process forgets its past exponentially fast.
For hidden Markov models, this forgetting property is closely related to the exponential stability of the optimal filter, see for instance \cite{legland2000forgetting, gerencser2007forgetting, douc2004asymptotic, douc2009forgetting}.
The last assumption is that the likelihood of the true process has sub-polynomial tails, or equivalently a finite moment.
None of these assumptions are specific to HMMs, thus making our result applicable to the misspecified setting.

To approximate a large class of true distributions, we consider nonparametric HMMs, where the parameters are not described by a finite dimensional space. For instance, one may consider HMMs with arbitrary number of states and arbitrary emission distributions.
Computing a maximizer of the likelihood directly in a nonparametric model may be hard or result in overfitting.
The model selection approach offers a way to circumvent this issue. It consists in considering a countable family of parametric sets $(S_M)_{M \in \Mcal}$\---the \emph{models}\---and selecting one of them. The larger the union of all models, the more distributions are approximated. Several criteria can be used to select the model, such as bootstrap, cross validation (see for instance \cite{arlot2010survey}) or penalization (see for instance \cite{Mas07}). We use a penalized criterion, which consists in maximizing the function
\begin{equation*}
(S, \theta \in S) \longmapsto \frac{1}{n} \log p_\theta(Y_1, \dots, Y_n) - \pen_n(S),
\end{equation*}
where $p_\theta$ is the density of $(Y_1, \dots, Y_n)$ under the parameter $\theta$ and the penalty $\pen$ only depends on the model $S$ and the number of observations $n$.

Assume that the emission distributions of the HMMs\---that is the distribution of the observations conditionally to the hidden states\---are absolutely continuous with respect to some known probability measure, and call \emph{emission densities} their densities with respect to this measure.
The tail assumption ensures that the emission densities have sub-polynomial tail:
\begin{equation*}
\forall v \geq e, \qquad \Pbb^* \left( \sup_{\gamma} \gamma(Y_1) \geq v^{\Csigma \log n} \right) \leq \frac{1}{v},
\end{equation*}
where the supremum is taken over all emission densities $\gamma$ in the models and for some constant $\Csigma > 0$. For instance, this assumption holds when all densities are upper bounded by $e^{\Csigma \log n}$. A key remark at this point is the dependency of the exponent with $n$: we allow the models to depend on the sample size. Typically, taking a larger sample makes it possible to consider larger models.

To stabilize the log-likelihood, we modify the models in the following way. First, only keep HMMs whose transition matrix have entries that are neither too small nor too large: when the HMM has $K$ hidden states, the entries of the transition matrix should belong to the interval $[K / (\CB \log n), K \CB \log n]$ for some constant $\CB > 0$.
Then, replace the emission densities $\gamma$ by a convex combination of the original emission densities and of the dominating measure $\lambda$ with a weight that decreases polynomially with the sample size. In other words, replace $\gamma$ by $(1 - n^{-a}) \gamma + n^{-a} \lambda$ for some $a > 0$. Taking $a > 1$ ensures that the component $\lambda$ is asymptotically negligible. Any $a > 0$ works, but the constants of the oracle inequality depend on it.

A simplified version of our main result (Theorem~\ref{th_oracle_simplifie}) is the following oracle inequality: there exist constants $A$ and $n_0$ such that if the penalty is large enough, the penalized maximum likelihood estimator $\hat{\theta}_n$ satisfies for all $t \geq 1$, $\eta \in (0,1)$ and $n \geq n_0$, with probability larger than $1 - e^{-t} - n^{-2}$:
\begin{equation*}
\Kbf(\hat{\theta}_n) \leq (1+\eta) \inf_{\dim(S) \leq n} \left\{ \inf_{\theta \in S}\Kbf(\theta) + 2\pen_n(S) \right\} + \frac{A}{\eta} t \frac{(\log n)^{10}}{n},
\end{equation*}
where $\Kbf(\theta)$ can be seen as a Kullback-Leibler divergence between the distributions $\Pbb^*$ and $\Pbb_\theta$. In other words, the estimator recovers the best approximation of the true distribution within the model, up to the penalty and the residual term.

In the case where the true distribution is a HMM, it is possible to quantify the approximation error $\inf_{\theta \in S}\Kbf(\theta)$. Using the results of \cite{kruijer2010adaptive}, we show that the above oracle inequality is optimal in the minimax sense\---up to logarithmic factors\---for real-valued HMMs, see Corollary~\ref{cor_mixture_minimax_estimation}. This is done by taking HMMs whose emission densities are mixtures of exponential power distributions\---which include Gaussian mixtures as a special case.

The paper is organized as follows.
We detail the framework of the article in Section~\ref{sec_notations_assumptions}. In particular, Section~\ref{sec_star_assumptions} describes the assumptions on the true distribution, Section~\ref{sec_model_assumptions} presents the assumptions on the model and Section~\ref{sec_prop_loglikelihood} introduces the Kullback Leibler criterion used in the oracle inequality.
Our main results are stated in Section~\ref{sec_main}. Section~\ref{sec_oracle} contains the oracle inequality and Section~\ref{sec_minimax} shows how it can be used to show minimax adaptivity for real-valued HMMs.
Section~\ref{sec_perspectives} lists some perspectives for this work.

One may wish to relax our assumptions depending on the setting. For instance, one could want to change the tail conditions or the rate of forgetting. We give an overview of the key steps of the proof of our oracle inequality in Section~\ref{sec_toute_la_preuve_oracle} to make it easier to adapt our result.

Some proofs are postponed the Appendices. Appendix~\ref{sec_proof_minimax} contains the proof of the minimax adaptivity result and Appendix~\ref{sec_proof_control_barnuk} contains the proof of the main technical lemma of Section~\ref{sec_toute_la_preuve_oracle}.

\section{Notations and assumptions}
\label{sec_notations_assumptions}

We will use the following notations:
\begin{itemize}[noitemsep]
\item $a \vee b$ is the maximum of $a$ and $b$, $a \wedge b$ the minimum;
\item For $x \in \Rbb$, we write $x^+ = x \vee 0$;
\item $\Nbb^* = \{1, 2, 3, \dots \}$ is the set of positive integers;
\item For $K \in \Nbb^*$, we write $[K] = \{1, 2, \dots, K\}$;
\item $Y_a^b$ is the vector $(Y_a, \dots, Y_b)$;
\item $\Lbf^2(A,\Acal,\mu)$ is the set of measurable and square integrable functions defined on the measured space $(A,\Acal,\mu)$. We write $\Lbf^2(A,\mu)$ when the sigma-field is not ambiguous;
\item $\log$ is the inverse function of the exponential function $\exp$.
\end{itemize}

\subsection{Hidden Markov models}

Finite state space hidden Markov models (HMM in short) are stochastic processes $(X_t, Y_t)_{t \geq 1}$ with the following properties. The \emph{hidden state} process $(X_t)_t$ is a Markov chain taking value in a finite set $\Xcal$ (the \emph{state space}). We denote by $K$ the cardinality of $\Xcal$, and $\pi$ and $\Qbf$ the initial distribution and transition matrix of $(X_t)_t$ respectively. The \emph{observation} process $(Y_t)_t$ takes value in a polish space $\Ycal$ (the \emph{observation space}) endowed with a Borel probability measure $\lambda$. The observations $Y_t$ are independent conditionally to $(X_t)_t$ with a distribution depending only on $X_t$. In the following, we assume that the distribution of $Y_t$ conditionally to $\{X_t = x\}$ is absolutely continuous with respect to $\lambda$ with density $\gamma_x$. We call $\gamma = (\gamma_x)_{x \in \Xcal}$ the \emph{emission densities}.

Therefore, the parameters of a HMM are its number of hidden states $K$, its initial distribution $\pi$ (the distribution of $X_1$), its transition matrix $\Qbf$ and its emission densities $\gamma$.
When appropriate, we write $p_{(K,\pi,\Qbf,\gamma)}$ the density of the process with respect to the dominating measure under the parameters ${(K,\pi,\Qbf,\gamma)}$. For a sequence of observations $Y_1^n$, we denote by $l_n(K,\pi,\Qbf,\gamma)$ the associated log-likelihood under the parameters ${(K,\pi,\Qbf,\gamma)}$, defined by
\begin{equation*}
l_n(K,\pi,\Qbf,\gamma) = \log p_{(K,\pi,\Qbf,\gamma)}(Y_1^n).
\end{equation*}

We denote by $\Pbb^*$ the true (and unknown) distribution of the process $(Y_t)_t$, $\Ebb^*$ the expectation under $\Pbb^*$, $p^*$ the density of $\Pbb^*$ under the dominating measure and $l_n^*$ the log-likelihood of the observations under $\Pbb^*$. Let us stress that this distribution may not be generated by a finite state space HMM.

\subsection{The model selection estimator}

Let $(S_{K,M,n})_{K \in \Nbb^*, M \in \Mcal}$ be a family of parametric models such that for all $K \in \Nbb^*$ and $M \in \Mcal$, the parameters ${(K,\pi,\Qbf,\gamma)} \in S_{K,M,n}$ correspond to HMMs with $K$ hidden states.
Note that the models $S_{K,M,n}$ may depend on the number of observations $n$. Let us see two ways to construct such models.

\begin{description}
\item[Mixture densities.] Let $\{ f_\xi \}_{\xi \in \Xi}$ be a parametric family of probability densities. Let $\Mcal \subset \Nbb^*$. We choose $S_{K,M,n}$ to be the set of parameters $(K, \pi, \Qbf, \gamma)$ such that $\pi$ and $\Qbf$ are the initial distribution and transition matrix of a Markov chain on $[K]$ and for all $x \in [K]$, $\gamma_x$ is a convex combination of $M$ elements of $\{ f_\xi \}_{\xi \in \Xi}$.

\item[$\Lbf^2$ densities.] Let $(E_M)_{M \in \Mcal}$ be a family of finite dimensional subspaces of $\Lbf^2(\Ycal, \lambda)$. We choose $S_{K,M,n}$ to be the set of parameters $(K, \pi, \Qbf, \gamma)$ such that $\pi$ and $\Qbf$ are the initial distribution and transition matrix of a Markov chain on $[K]$ and for all $x \in [K]$, $\gamma_x$ is a probability density such that $\gamma_x = g \vee 0$ for a function $g \in E_M$.
\end{description}
For all $K \in \Nbb^*$ and $M \in \Mcal$, we define the maximum likelihood estimator on $S_{K,M,n}$:
\begin{equation*}
{(K, \hat\pi_{K,M,n}, \hat\Qbf_{K,M,n}, \hat\gamma_{K,M,n})}
	\in \underset{(K,\pi,\Qbf,\gamma) \in S_{K,M,n}}{\argmax} \; \frac{1}{n} l_n(K,\pi,\Qbf,\gamma).
\end{equation*}

Since the true distribution does not necessarily correspond to a parameter of $S_{K,M,n}$, taking a larger model $S_{K,M,n}$ will reduce the bias of the estimator $(K, \hat\pi_{K,M,n}, \hat\Qbf_{K,M,n}, \hat\gamma_{K,M,n})$. However, larger models will make the estimation more difficult, resulting in a larger variance. This means one has to perform a bias-variance tradeoff to select a model with a reasonable size. To do so, we select a number of states $\hat{K}_n$ among a set of integers $\Kcal_n$ and a model index $\hat{M}_n$ among a set of indices $\Mcal_n$ such that the penalized log-likelihood is maximal:
\begin{equation*}
(\hat{K}_n, \hat{M}_n) \in \underset{K \in \Kcal_n, M \in \Mcal_n}{\argmax} \left( \frac{1}{n} l_n(K, \hat\pi_{K,M,n}, \hat\Qbf_{K,M,n}, \hat\gamma_{K,M,n}) - \pen_n(K,M) \right)
\end{equation*}
for some penalty $\pen_n$ to be chosen.

In the following, we use the following notations.
\begin{itemize}
\item $\Sbf_n := \bigcup_{K \in \Kcal_n, M \in \Mcal_n} S_{K,M,n}$ is the set of all parameters involved with the construction of the maximum likelihood estimator;

\item $S_{K,M,n}^{(\gamma)} = \{ \gamma \, | \, (K,\pi,\Qbf,\gamma) \in S_{K,M,n} \}$ is the set of density vectors from the model $S_{K,M,n}$. $\Sbf_n^{(\gamma)}$ is defined in the same way.
\end{itemize}

\subsection{Assumptions on the true distribution}
\label{sec_star_assumptions}

In this section, we introduce the assumptions on the true distribution of the process $(Y_t)_{t \geq 1}$. We assume that $(Y_t)_{t \geq 1}$ is stationary, so that one can extend it into a process $(Y_t)_{t \in \Zbb}$.

\begin{description}
\labitem{{[A$\star$tail]}}{Astar_tail}
There exists $\delta > 0$ such that 
\begin{equation*}
M_\delta := \sup_{i,k} \Ebb^*[ (p^*(Y_i | Y_{i-k}^{i-1}))^\delta ] < \infty.
\end{equation*}
\end{description}
This assumption ensures that the true log-density rarely takes extreme values (see Lemma~\ref{lem_tail_moments}).
\begin{description}
\labitem{{[A$\star$forget
]}}{Astar_forgetting} There exist two constants $C_* > 0$ and $\rho_* \in (0,1)$ such that for all $i \in \Zbb$, for all $k,k' \in \Nbb^*$ and for all $y_{i-(k \vee k')}^i \in \Ycal^{(k \vee k')+1}$,
\begin{equation*}
| \log p^*(y_i | y_{i-k}^{i-1}) - \log p^*(y_i | y_{i-k'}^{i-1}) | \leq C_* \rho_*^{k \wedge k' - 1}
\end{equation*}
\end{description}
Let us recall the definition of the $\rho$-mixing coefficient. Let $(\Omega, \Fcal, P)$ be a measured space and $\Acal \subset \Fcal$ and $\Bcal \subset \Fcal$ be two sigma-fields. Let
\begin{equation*}
\rho_\mix(\Acal, \Bcal) = \underset{g \in \Lbf^2(\Omega,\Bcal,P)}{\sup_{f \in \Lbf^2(\Omega,\Acal,P)}} | \Corr(f,g) |.
\end{equation*}

The $\rho$-mixing coefficient of $(Y_t)_t$ is defined by
\begin{equation*}
\rho_\mix(n) = \rho_\mix(\sigma(Y_i, i \geq n), \sigma(Y_i, i \leq 0)).
\end{equation*}

\begin{description}
\labitem{{[A$\star$mix
]}}{Astar_mixing} There exist two constants $c_* > 0$ and $n_* \in \Nbb^*$ such that
\begin{equation*}
\forall n \geq n_*, \quad \rho_\mix(n) \leq 4 e^{- c_* n}.
\end{equation*}
\end{description}
Assumption~\ref{Astar_forgetting} ensures that the process forgets its initial distribution exponentially fast. 
This assumption is especially useful for truncating the dependencies in the likelihood.
\ref{Astar_mixing} is a usual mixing assumption and is used to obtain Bernstein-like concentration inequalities. Note that~\ref{Astar_mixing} implies that the process $(Y_t)_{t \geq 1}$ is ergodic.

Even if~\ref{Astar_forgetting} is analog to a $\psi$-mixing condition (see \cite{bradley2005strongmixingsurvey} for a survey on mixing conditions) and is proved using the same tool~\ref{Astar_mixing} in hidden Markov models\---namely the geometric ergodicity of the hidden state process\---these two assumptions are different in general. For instance, a Markov chain always satisfies~\ref{Astar_forgetting} but not necessarily~\ref{Astar_mixing}. Conversely, there exist processes satisfying~\ref{Astar_mixing} but not~\ref{Astar_forgetting}.

\begin{lemma}
\label{lemma_exempleAppli_Astar_HMMcompact}
Assume that $(Y_t)_t$ is generated by a HMM with a compact metric state space $\Xcal$ (not necessarily finite) endowed with a Borel probability measure $\mu$. Write $\Qcal^*$ its transition kernel and assume that $\Qcal^*$ admits a density with respect to $\mu$ that is uniformly lower bounded and upper bounded by positive and finite constants $\sigma^*_-$ and $\sigma^*_+$. Write $(\gamma^*_x)_{x \in \Xcal}$ its emission densities and assume that they satisfy $\int \gamma^*_x(y) \mu(dx) \in (0, +\infty)$ for all $y \in \Ycal$.

Then \ref{Astar_forgetting} and \ref{Astar_mixing} hold by taking $\rho_* = 1 - \frac{\sigma^*_-}{\sigma^*_+}$, $C_* = \frac{1}{1 - \rho_*}$, $c_* = \frac{- \log (1 - \sigma_-^*)}{2}$ and $n_* = 1$.
\end{lemma}

\begin{proof}
This lemma follows from the geometric ergodicity of the HMM.

For \ref{Astar_forgetting}, see for instance \cite{douc2004asymptotic}, proof of Lemma 2.

For \ref{Astar_mixing}, the Doeblin condition implies that for all distributions $\pi$ and $\pi'$ on $\Xcal$,
\begin{align*}
\int | p^*(X_n = x | X_0 \sim \pi) - p^*(X_n = x | X_0 \sim \pi') | \mu(dx) \leq (1 - \sigma_-^*)^n \| \pi - \pi'\|_1.
\end{align*}

Let $A \in \sigma(Y_t, t \geq k)$ and $B \in \sigma(Y_t, t \leq 0)$ such that $\Pbb^*(B) > 0$. Taking $\pi$ the stationary distribution of $(X_t)_t$ and $\pi'$ the distribution of $X_0$ conditionally to $B$ in the above equation implies
\begin{align*}
| \Pbb^*(A | B) - \Pbb^*(A) |
	&= \left| \int \Pbb^*(A | X_n = x) (p^*(X_n = x) - p^*(X_n = x | B)) \mu(dx) \right| \\
	&\leq \int | p^*(X_n = x) - p^*(X_n = x | B) | \mu(dx) \\
	&\leq 2 (1 - \sigma_-^*)^n.
\end{align*}

Therefore, the process $(Y_t)_{t \geq 1}$ is $\phi$-mixing with $\phi_\mix(n) \leq 2 (1 - \sigma_-^*)^n$, so that it is $\rho$-mixing with $\rho_\mix(n) \leq 2 (\phi_\mix(n))^{1/2} \leq 2 \sqrt{2} (1 - \sigma_-^*)^{n/2}$ (see e.g. \cite{bradley2005strongmixingsurvey} for the definition of the $\phi$-mixing coefficient and its relation to the $\rho$-mixing coefficient). One can check that the choice of $c_*$ and $n_*$ allows to obtain \ref{Astar_mixing} from this inequality.
\end{proof}

\subsection{Model assumptions}
\label{sec_model_assumptions}

We now state the assumptions on the models. Let us recall that the distribution of the observed process is not assumed to belong to one of these models.

Consider a family of models $(S_{K,M,n})_{K \in \Nbb^*, M \in \Mcal}$ such that for each $K$ and $M$, the elements of $S_{K,M,n}$ are of the form $(K, \pi, \Qbf, \gamma)$ where $\pi$ is a probability density on $[K]$, $\Qbf$ is a transition matrix on $[K]$ and $\gamma$ is a vector of $K$ probability densities on $\Ycal$ with respect to $\lambda$.

The first assumption is standard in maximum likelihood estimation. It ensures that the process forgets the past exponentially fast, which implies that the difference between the normalized log-likelihood $\frac{1}{n} l_n$ and its limit converges to zero with rate $1/n$ in supremum norm.
\begin{description}
\labitem{[Aergodic]}{Aergodic} There exists $\Csigma \geq 1$ such that for all $(K,\pi,\Qbf,\gamma) \in \Sbf_n$,
\begin{eqnarray*}
\forall x,x' \in [K], \qquad (\Csigma \log n)^{-1} \leq K \Qbf(x,x') \leq \Csigma \log n \\
\text{and} \qquad 
\forall x \in [K], \qquad (\Csigma \log n)^{-1} \leq K \pi(x) \leq \Csigma \log n.
\end{eqnarray*}
\end{description}
For all $\gamma \in \Sbf_n^{(\gamma)}$ and $y \in \Ycal$, let
\begin{equation*}
b_\gamma(y) = \log \left(K^{-1} \sum_{x} \gamma_x(y) \right).
\end{equation*}

When $(K,\pi,\Qbf,\gamma) \in \Sbf_n$, assumption \ref{Aergodic} implies that under the parameters $(K,\pi,\Qbf,\gamma)$, for all $x \in [K]$, the probability to jump to state $x$ at time $t$ is at least $(\Csigma \log n)^{-1} K^{-1}$, whatever the past may be. This implies that the density $p_{(K,\pi,\Qbf,\gamma)}(Y_t | Y_1^{t-1})$ is lower bounded by $(\Csigma \log n)^{-1} K^{-1} \sum_x \gamma_x(Y_t)$. For the same reason, it is upper bounded by $\Csigma (\log n) K^{-1} \sum_x \gamma_x(Y_t)$. Thus, it is enough to bound $b_\gamma$ to control $p_{(K,\pi,\Qbf,\gamma)}$ without having to handle the dependency in past observations.

The following assumption ensures that the log-likelihood rarely takes extreme values.
\begin{description}
\labitem{[Atail]}{Atail}
There exists $\CB \geq 1$ such that
\begin{equation*}
\forall u \geq 1, \quad \Pbb^* \left[ \sup_{\gamma \in \Sbf_n^{(\gamma)}} \left| b_\gamma(Y_1) \right| \geq \CB (\log n) u \right] \leq e^{-u}.
\end{equation*}
\end{description}
In practice, it is enough to check the upper deviations, as shown in the following lemma.
\begin{lemma}
\label{lemma_atail}
Assume that there exists $C \geq 1$ such that
\begin{equation*}
\forall u \geq 1, \quad \Pbb^* \left[ \sup_{\gamma \in \Sbf_n^{(\gamma)}} b_\gamma(Y_1) \geq C (\log n) u \right] \leq e^{-u}.
\end{equation*}

Consider a new model where all $\gamma$ are replaced by $\gamma' = (1 - n^{-a}) \gamma + n^{-a}$ for a fixed constant $a > 0$. Then \ref{Atail} holds for this new model with $\CB = C \vee a$.
\end{lemma}

Changing the densities as in the lemma amounts to adding a mixture component (with weight $n^{-a}$ and distribution $\lambda$) to the emission densities to make sure that they are uniformly lower bounded. We shall see in the following that if $a \geq 1$, then this additional component changes nothing to the approximation properties of the models, see the proof of Corollary~\ref{cor_mixture_minimax_estimation}. This is in agreement with the fact that this component is asymptotically never observed as soon as $a > 1$.

The following assumption means that as far as the bracketing entropy is concerned, the set of emission densities of the model $S_{K,M,n}$ behaves like a parametric model with dimension $m_M$.
\begin{description}
\labitem{[Aentropy]}{Aentropy}
There exists a function $(M, K, D, n) \longmapsto \Caux(M, K, D, n) \geq 1$ and a sequence $(m_M)_{M \in \Mcal} \in \Nbb^\Mcal$ such that for all $\delta > 0$, $M$, $K$, $n$ and $D$,
\begin{multline}
\label{eq_controle_entropie_classe_gamma}
N\left(\left\{ y \mapsto \gamma_x(y) \one_{\underset{\gamma' \in \Sbf_n^{(\gamma)}}{\sup} | b_{\gamma'}(y) | \leq D} \right\}_{\gamma \in S_{K,M,n}^{(\gamma)}, x \in [K]}, d_\infty, \delta \right) \\
	\leq \max \left( \frac{\Caux(M, K, D, n)}{\delta} , 1\right)^{m_M},
\end{multline}
where $d_\infty$ is the supremum norm distance and $N\left(A, d, \epsilon \right)$ is the smallest number of brackets of size $\epsilon$ for the distance $d$ needed to cover $A$. Let us recall that the bracket $[a,b]$ is the set of functions $f$ such that $a(\cdot) \leq f(\cdot) \leq b(\cdot)$, and that the size of the bracket $[a,b]$ is $d(a,b)$.
\end{description}
Note that we allow the models to depend on the sample size $n$, which can make $\Caux$ grow to infinity with $n$. The following assumption ensures that the models do not grow absurdly fast.
\begin{description}
\labitem{[Agrowth]}{Agrowth} There exist $\zeta > 0$ and $n_\text{growth}$ such that for all $n \geq n_\text{growth}$,
\begin{equation*}
\sup_{K, M \text{ s.t. } K \leq n \text{ and } m_M \leq n} \log \Caux(M,K, 3\CB (\log n)^2, n) \leq n^\zeta.
\end{equation*}
\end{description}
A typical way to check \ref{Aentropy} is to use a parametrization of the emission densities, for instance a lipschitz application $[-1,1]^{m_M} \longrightarrow S^{(\gamma)}_{K,M,n}$. This reduces the construction of a bracket covering on $S^{(\gamma)}_{K,M,n}$ to the construction of a bracket covering of the unit ball of $\Rbb^{m_M}$. In this case, $\Caux$ depends on the lipschitz constant of the parametrization. Baring models $\Sbf_n^{(\gamma)}$ that grow so fast with respect to $n$ that \ref{Aentropy} becomes essentially meaningless, \ref{Agrowth} is usually immediately checked once \ref{Aentropy} is established. An example of this approach is given in Section~\ref{sec_minimax} for mixtures of exponential power distributions.

\subsection{Limit and properties of the normalized log-likelihood}
\label{sec_prop_loglikelihood}

In this section, we focus on the convergence of the normalized log-likelihood.
\begin{lemma}[\cite{barron1985entropyrate}]
\label{lemma_Barron}
Assume that the process $(Y_t)_{t \geq 1}$ is ergodic, then there exists a quantity ${l^* > -\infty}$ such that
\begin{equation*}
\frac{1}{n} l_n^* \underset{n \rightarrow \infty}{\longrightarrow} l^* \quad \text{a.s.}
\end{equation*}
and
\begin{equation*}
l^* = \lim_{n \rightarrow \infty} \Ebb^*[ \log p^*(Y_n | Y_1^{n-1}) ].
\end{equation*}
\end{lemma}

The second result follows from Theorem 2 of \cite{leroux92MLEHMM}.
\begin{lemma}[\cite{leroux92MLEHMM}]
\label{lemma_Leroux}
Let $K$ be a positive integer, $\gamma$ a vector of $K$ probability densities, $\Qbf$ a transition matrix of size $K$ and $\pi$ a probability measure on $[K]$. 
Assume that the process $(Y_t)_{t \geq 1}$ is ergodic and that $\pi(x) > 0$ and $\Ebb^* |\log \gamma_x(Y_1) | < +\infty$ for all $x \in [K]$.

Then there exists a finite quantity $l(K,\Qbf,\gamma)$ which does not depend on $\pi$ such that
\begin{equation*}
\frac{1}{n} l_n(K,\pi,\Qbf,\gamma) \underset{n \rightarrow \infty}{\longrightarrow} l(K,\Qbf,\gamma) \quad \Pbb^*\text{-a.s.} \text{ and in } \Lbf^1(\Pbb^*).
\end{equation*}

In particular, $l(K,\Qbf,\gamma) = \lim_n \Ebb[\frac{1}{n} l_n(K,\pi,\Qbf,\gamma)]$.
\end{lemma}

When appropriate, we define $\Kbf(K, \Qbf, \gamma)$ by
\begin{align*}
\Kbf(K, \Qbf, \gamma) := l^* - l(K, \Qbf, \gamma).
\end{align*}

Note that $\Kbf(K,\Qbf,\gamma) \geq 0$ since it is the limit of a sequence of Kullback-Leibler divergences: under the assumptions of Lemma~\ref{lemma_Leroux},
\begin{align*}
\Kbf(K, \Qbf, \gamma) = \lim_{n \rightarrow \infty} \frac{1}{n} KL ( \Pbb^*_{Y_1^n} \| \Pbb_{Y_1^n | (K,\pi,\Qbf,\gamma)})
\end{align*}
where $\Pbb^*_{Y_1^n}$ (respectively $\Pbb_{Y_1^n | (K,\pi,\Qbf,\gamma)}$) is the distribution of $Y_1^n$ under $\Pbb^*$ (respectively $\Pbb_{(K,\pi,\Qbf,\gamma)}$).
We will see in the proofs that with some notation abuses:
\begin{align*}
\Kbf(K,\Qbf,\gamma) &= \Ebb^* \left[
	\log \left( \frac{p^*(Y_1 | Y_{-\infty}^0)}{p_{(K,\Qbf,\gamma)}(Y_1 | Y_{-\infty}^0)} \right)
\right] \\
	&= \Ebb^*_{Y_{-\infty}^0} \left[ KL( \Pbb^*_{Y_1 | Y_{-\infty}^0} \| \Pbb_{Y_1 | Y_{-\infty}^0, (K,\Qbf,\gamma)} ) \right].
\end{align*}

Thus, $\Kbf(K,\Qbf,\gamma)$ can be seen as a Kullback Leibler divergence that measures the difference between the distribution of $Y_1$ conditionally to the whole past under the parameter $(K,\Qbf,\gamma)$ and under the true distribution. In a way, it is a prediction error under the parameter $(K,\Qbf,\gamma)$.

In the particular case where the true distribution of $(Y_t)_t$ comes from a finite state space hidden Markov model, $\Kbf$ characterizes the true parameters, up to permutation of the hidden states, provided the emission densities are all distinct and the transition matrix is invertible, as shown in the following result.
\begin{lemma}[\cite{AH14}, Theorem 5]
\label{lemma_identifiabilite_Kbf}
Assume $(Y_t)_t$ is generated by a finite state space HMM with parameters $(K^*, \pi^*, \Qbf^*, \gamma^*)$. Assume $\Qbf^*$ is invertible and ergodic, that the emission densities $(\gamma^*_x)_{x \in [K^*]}$ are all distinct and that $\Ebb^* \left[ (\log \gamma^*_x(Y_1))^+ \right] < \infty$ for all $x \in [K^*]$ (so that $l^* < \infty$).

Then for all $K \leq K^*$, for all transition matrices $\Qbf$ of size $K$ and for all $K$-uples of probability densities $\gamma$, ${\Kbf(K,\Qbf,\gamma) = 0}$ if and only if $(K,\Qbf,\gamma) = (K^*, \Qbf^*, \gamma^*)$ up to permutation of the hidden states.
\end{lemma}

\section{Main results}
\label{sec_main}

\subsection{Oracle inequality for the prediction error}
\label{sec_oracle}

The following theorem states an oracle inequality on the prediction error of our estimator. It shows that with high probability, our estimator performs as well as the best model of the class in terms of Kullback Leibler divergence, up to a multiplicative constant and up to an additive term decreasing as $\frac{(\log n)^{\cdots}}{n}$, provided the penalty is large enough.

\begin{theorem}
\label{th_oracle_simplifie}
Assume \ref{Astar_forgetting}, \ref{Astar_mixing}, \ref{Astar_tail},
\ref{Aergodic}, \ref{Atail}, \ref{Aentropy} and \ref{Agrowth} hold.

Let $(w_M)_{M \in \Mcal}$ be a nonnegative sequence such that $\sum_{M \in \Mcal} e^{-w_M} \leq e-1$.
For all $K$ and $M$, let
\begin{equation*}
{(K, \hat\pi_{K,M,n}, \hat\Qbf_{K,M,n}, \hat\gamma_{K,M,n})}
	\in \underset{(K,\pi,\Qbf,\gamma) \in S_{K,M,n}}{\argmax} \; \frac{1}{n} l_n(K,\pi,\Qbf,\gamma),
\end{equation*}
\begin{equation*}
(\hat{K}, \hat{M}) \in \underset{M \text{ s.t. } m_M \leq n}{\underset{K \leq n}{\argmax}} \left( \frac{1}{n} l_n(K, \hat\pi_{K,M,n}, \hat\Qbf_{K,M,n}, \hat\gamma_{K,M,n}) - \pen_n(K,M) \right)
\end{equation*}
and let
$
(\hat{K}, \hat{\pi}, \hat{\Qbf}, \hat{\gamma}) = (\hat{K}, \hat\pi_{\hat{K},\hat{M},n}, \hat\Qbf_{\hat{K},\hat{M},n}, \hat\gamma_{\hat{K},\hat{M},n})
$
be the nonparametric maximum likelihood estimator.

Then there exist constants $A$ and $C_\pen$ depending only on $\Csigma$, $\CB$, $n_*$ and $c_*$ and a constant $n_0$ depending only on $\Csigma$, $\CB$, $n_*$, $\zeta$, $n_\text{growth}$, $C_*$, $\rho_*$, $\delta$ and $M_\delta$ such that for all
$n \geq n_0$, $t \geq 1$ and $\eta \leq 1$, with probability at least $1 - e^{-t} - 2 n^{-2}$,
\begin{multline*}
\Kbf(\hat{K}, \hat{\Qbf}, \hat{\gamma})
	\leq (1+\eta) \underset{M \text{ s.t. } m_M \leq n}{\underset{K \leq n}{\inf}} \Bigg\{ \inf_{(K,\pi,\Qbf,\gamma) \in S_{K,M,n}}\Kbf(K,\Qbf,\gamma) \\
	 + 2 \pen_n(K,M) \Bigg\}
		+ \frac{A}{\eta} t \frac{(\log n)^{10}}{n}
\end{multline*}
as soon as
\begin{multline*}
\pen_n(K,M) \geq
	\frac{C_\pen}{\eta} \frac{(\log n)^{10}}{n} \Bigg\{ w_M
	+ (\log n)^{4} ( m_M K + K^2 - 1) \\
		\times \left( (\log n)^3 \log \log n + \log \Caux(M, K, 3 \CB (\log n)^2, n) \right) \Bigg\}.
\end{multline*}
\end{theorem}

The proof of this theorem is presented in Section~\ref{sec_toute_la_preuve_oracle}. Its structure and main steps are detailed in Section~\ref{sec_structure_preuve_oracle}, and the proof of these steps are gathered in Section~\ref{sec_preuves_techniques_oracle}.

Note that this theorem is not specific to one choice of the parametric models $S_{K,M,n}$: one may choose the type of model that suits the density one wants to estimate best. In the following section, we use mixture models to estimate densities when $\Ycal$ is unbounded. If $\Ycal$ is compact, we could use $\Lbf^2$ spaces and this oracle inequality would still hold.

The powers of $\log n$ 
come from:
\begin{itemize}
\item The limitation of the dependency to the $\log n$ most recent observations,

\item The dependency of the bounds $\Csigma \log n$ and $\CB \log n$ on $n$ in assumptions \ref{Aergodic} and \ref{Atail},

\item Truncating the emission log-densities (possible thanks to assumptions \ref{Atail} and \ref{Astar_tail}),

\item The use of a Bernstein inequality for exponentially $\alpha$-mixing processes.
\end{itemize}

\subsection{Minimax adaptive estimation using location-scale mixtures}
\label{sec_minimax}

In this section, we show that the oracle inequality of Theorem~\ref{th_oracle_simplifie} allows to construct an estimator that is adaptive and minimax up to logarithmic factors when the observations are generated by a finite state space hidden Markov model. To do so, we consider models whose emission densities are finite mixtures of exponential power distributions, and use an approximation result by \cite{kruijer2010adaptive}.

Assume that $(Y_t)_{t \geq 1}$ is generated by a stationary HMM with parameters $(K^*, \Qbf^*, \gamma^*)$, which we call the true parameters. Without loss of generality, we identify the true hidden state space with $[K^*]$. We consider the case $\Ycal = \Rbb$ endowed with the probability $\lambda$ with density $\Gref : y \longmapsto (\pi(1 + y^2))^{-1}$ with respect to the Lebesgue measure.

In order to quantify the approximation error by location-scale mixtures, we use the following assumptions from \cite{kruijer2010adaptive}.
\begin{description}
\item[(C1)] \emph{Smoothness}. For all $x \in [K^*]$, $\log (\gamma^*_x \Gref)$ is locally $\beta$-Hölder with $\beta > 0$, i.e. there exist a polynomial $L$ and a constant $R > 0$ such that if $r$ is the largest integer smaller than $\beta$, one has for all $x \in [K^*]$, 
\begin{multline*}
\forall y, y' \text{ s.t. } |y - y'| \leq R, \\
	\left| \frac{\partial^r \log(\gamma^*_x \Gref)}{\partial y^r}(y) - \frac{\partial^r \log(\gamma^*_x \Gref)}{\partial y^r}(y') \right| \leq r! L(y) |y-y'|^{\beta - r}.
\end{multline*}

\item[(C2)] \emph{Moments}. There exists $\epsilon > 0$ such that for all $x \in [K^*]$, 
\begin{gather*}
\forall j \in \{1, \dots, r\}, \quad \int \left| \frac{\partial^j \log(\gamma^*_x \Gref)}{\partial y^j}(y) \right|^{\frac{2\beta + \epsilon}{j}} (\gamma^*_x \Gref)(y) d y < \infty
\\
\int L(y)^{\frac{2\beta + \epsilon}{\beta}} (\gamma^*_x \Gref)(y) d y < \infty
\end{gather*}

\item[(C3)] \emph{Tail}. There exist positive constants $c$ and $\tau$ such that for all $x \in [K^*]$, 
\begin{equation*}
\gamma^*_x \Gref = O(e^{-c |y|^\tau}).
\end{equation*}

\item[(C4)] \emph{Monotonicity}. For all $x \in [K^*]$, $(\gamma^*_x \Gref)$ is positive and there exists $y_m < y_M$ such that for all $x \in [K^*]$, $(\gamma^*_x \Gref)$ is nondecreasing on $(-\infty, y_m)$ and nonincreasing on $(y_M, +\infty)$.
\end{description}

All these assumptions refer to the functions $(\gamma^*_x \Gref)$, which are the densities of the true emission distributions with respect to the Lebesgue measure. Hence, the choice of the dominating measure $\lambda$ does not matter as far as regularity conditions are concerned.

Note that \cite{kruijer2010adaptive} only assumed \textbf{(C3)} outside of a compact set. However, since the regularity assumption \textbf{(C1)} implies that $(\gamma^*_x \Gref)$ is continuous, one may assume \textbf{(C3)} for all $y$ without loss of generality.

It is important to note that even though we require some regularity on the emission densities, for instance through the polynomial $L$ and the constants $\beta$ and $\tau$, we do not need to know them to construct our estimator, thus making it adaptive.
\bigskip

We consider the following models. Let $p \geq 2$ be an even integer and
\begin{equation*}
\psi(y) = \frac{1}{2 \Gamma \left( 1 + \frac{1}{p} \right)} e^{-y^p}.
\end{equation*}

Let $\Mcal = \Nbb^*$. We take $S_{K,M,n}$ as the set of parameters $(K, \pi, \Qbf, \gamma)$ such that
\begin{itemize}
\item \ref{Aergodic} holds with $C_\sigma = 1$,

\item For all $x \in [K]$, there exist $(s_{x,1}, \dots, s_{x,M}) \in [\frac{1}{n} , n]^{M}$, $(\mu_{x,1}, \dots, \mu_{x,M}) \in [-n , n]^{M}$ and $w_x = (w_{x,1}, \dots, w_{x,M}) \in [0,1]^M$ such that $\sum_i w_{x,i} = 1$ and for all $y \in \Rbb$,
\begin{equation*}
\gamma_x(y) = \frac{1}{n^2}
	+ \left( 1 - \frac{1}{n^2} \right) \frac{1}{\Gref(y)}
		\sum_{i=1}^{M} w_{x,i} \frac{1}{s_{x,i}} \psi \left( \frac{y - \mu_{x,i}}{s_{x,i}} \right).
\end{equation*}
In other words, the emission densities are mixtures of $\lambda$ (with weight $n^{-2}$) and of $M$ translations and dilatations of $\psi$.
\end{itemize}

\begin{lemma}[Checking the assumptions]
\label{lemma_mixture_checking_assumptions}
Assume $\inf \Qbf^* > 0$, then:
\begin{itemize}
\item \ref{Astar_forgetting} and \ref{Astar_mixing} hold.

\item Assume \textbf{(C3)}, then \ref{Astar_tail} holds.


\item \ref{Atail} holds for all $n \geq 3$ by taking $\CB = 10$.

\item \ref{Aentropy} and \ref{Agrowth} hold for any $\zeta > 0$ by taking $m_M = 2M$ and $\Caux(M, K, D, n) = 4 p n^3$, for instance $\zeta = 2$ and $n_\text{growth} = 4p$.
\end{itemize}
\end{lemma}

\begin{proof}
The first point follows from Lemma~\ref{lemma_exempleAppli_Astar_HMMcompact}.
The second point follows from the fact that the densities $\gamma^*_x$ are uniformly bounded under \textbf{(C3)}.

See Section~\ref{sec_preuve_lemma_mixture_checking_assumptions} for the proof of the last two points.
\end{proof}

\begin{remark}
The results of this section remain the same when the weight of $\lambda$ in the emission densities of $S_{K,M,n}$ is allowed to be larger than $n^{-2}$ instead of being exactly $n^{-2}$.
\end{remark}

Lemma~4 from \cite{kruijer2010adaptive} implies the following result.
\begin{lemma}[Approximation rates]
\label{lemma_mixture_approximation_rates}
Assume \textbf{(C1)}-\textbf{(C4)} hold. Then there exists sequences of mixtures $(g_{M,x})_M$ for each $x \in [K^*]$ such that for $M$ large enough and all $n \geq M$, $(n^{-2} + (1 - n^{-2}) g_{M,x})_{x \in [K^*]} \in S_{K^*,M,n}^{(\gamma)}$ and
\begin{equation*}
\max_{x \in [K^*]} KL( \gamma^*_x \| g_{M,x}) = O(M^{-2\beta} (\log M)^{2\beta \frac{p}{\tau}}).
\end{equation*}
\end{lemma}

\begin{proof}
Proof in Section~\ref{sec_proof_lemma_mixture_approximation_rates}.
\end{proof}

\begin{corollary}[Minimax adaptive estimation rates]
\label{cor_mixture_minimax_estimation}
Assume \textbf{(C1)}-\textbf{(C4)} hold. Also assume that $\inf \Qbf^* > 0$. Then there exists a constant $C > 0$ such that for all $M \geq 3$ and $n \geq M$,
\begin{equation*}
	\inf_{(K^*,\pi,\Qbf,\gamma) \in S_{K^*,M,n}} \Kbf(K^*,\Qbf,\gamma)
		\leq C (\log n)^2 \left(\frac{1}{n} + M^{-2\beta} (\log M)^{2\beta \frac{p}{\tau}} \right)
\end{equation*}

Hence, using Theorem~\ref{th_oracle_simplifie} with $\pen_n(K,M) = (KM+K^2) (\log n)^{18}/n$, there exists a constant $C$ such that almost surely, there exists a (random) $n_0$ such that
\begin{align*}
\forall n \geq n_0, \quad
\Kbf(\hat{K}_n, \hat{\Qbf}_n, \hat{\gamma}_n)
	&\leq C n^{\frac{-2\beta}{2\beta+1}}  (\log n)^{18 + \frac{p}{\tau} - \frac{16 + \frac{p}{\tau}}{2\beta + 1}} \\
	&\leq C n^{\frac{-2\beta}{2\beta+1}} (\log n)^{18 + \frac{p}{\tau}}.
\end{align*}
\end{corollary}

\begin{proof}
Proof in Section~\ref{sec_proof_cor_mixture_minimax_estimation}.
\end{proof}

This result shows that our estimator reaches the minimax rate of convergence proved by \cite{maugis2013adaptive} for density estimation in Hellinger distance, up to logarithmic factors. Since estimating a density is the same thing as estimating a one-state HMM, this means that our result is adaptive and minimax up to logarithmic factors when $K^* = 1$. As far as we know, it is still unknown whether increasing the number of states improves the minimax rates of convergence. It seems reasonable to think that it doesn't, which would imply that our estimator is in general adaptive and minimax.

\section{Perspectives}
\label{sec_perspectives}

The main result of this paper is a guarantee that maximum likelihood estimators based on nonparametric hidden Markov models give sensible results even in the misspecified setting, and that their error can be controlled nonasymptotically.
Two properties of both the models and the true distributions are at the core of this result: a mixing property and a forgetting property, which can be seen as a local dependence property.

These two properties are not specific to hidden Markov models. Therefore, it is likely that our result can be generalized to many other models and distributions. To name a few, one could consider hidden Markov models with continuous state space as studied in~\cite{douc2001asymptotics} or~\cite{douc2011consistency}, or more generally partially observed Markov models, see for instance~\cite{douc2016posterior} and reference therein. Special cases of partially observed Markov models are HMMs with autoregressive properties~\citep{douc2004asymptotic} and models with time inhomogeneous Markov regimes~\citep{pouzo2016misspecified}. One could also consider hidden Markov fields~\citep{kunsch1995hiddenMarkovFields} and graphical models to generalize to more general distributions than time processes.

Another interesting approach is to consider other forgetting and mixing assumptions. For instance, \cite{legland2000forgetting} state a more general version of the forgetting assumption where the constant is replaced by an almost surely finite random variable, and \cite{gerencser2007forgetting} give conditions under which the moments of this random variable are finite. Other mixing and weak dependence conditions have also been introduced in the litterature with the hope of describing more general processes, see for instance~\cite{dedecker2007weakdependence}.

\section{Proof of the oracle inequality (Theorem~\ref{th_oracle_simplifie})}
\label{sec_toute_la_preuve_oracle}

\subsection{Overview of the proof}
\label{sec_structure_preuve_oracle}

By definition of $(\hat{K}, \hat{\pi}, \hat{\Qbf}, \hat{\gamma})$, one has for all $K \leq n$, for all $M$ such that $m_M \leq n$ and for all ${(K, \pi_{K,M}, \Qbf_{K,M}, \gamma_{K,M}) \in S_{K,M,n}}$:
\begin{multline*}
\frac{1}{n} l_n^* - \frac{1}{n} l_n(\hat{K}, \hat{\pi}, \hat{\Qbf}, \hat{\gamma})
	\leq \frac{1}{n} l_n^* - \frac{1}{n} l_n(K, \pi_{K,M}, \Qbf_{K,M}, \gamma_{K,M}) \\
		+ \pen_n(K,M) - \pen_n(\hat{K},\hat{M})
\end{multline*}
where $\hat{K}$ and $\hat{M}$ are the selected number of hidden states and model index respectively.

Let
\begin{equation*}
\nu(K, \pi, \Qbf, \gamma) := \left(\frac{1}{n} l_n^* - \frac{1}{n} l_n(K, \pi, \Qbf, \gamma)\right) - \Kbf(K, \Qbf, \gamma),
\end{equation*}
then
\begin{align*}
\Kbf(\hat{K}, \hat{\Qbf}, \hat{\gamma})
	\leq{} & \; \Kbf(K, \Qbf_{K,M}, \gamma_{K,M}) + 2 \pen_n(K,M) \\
		&+ \nu(K, \pi_{K,M}, \Qbf_{K,M}, \gamma_{K,M}) - \pen_n(K,M) \\
		&- \nu(\hat{K}, \hat{\pi}, \hat{\Qbf}, \hat{\gamma}) - \pen_n(\hat{K},\hat{M}).
\end{align*}

Now, assume that with high probability, for all $K$, $M$ and $(K,\pi,\Qbf,\gamma) \in S_{K,M,n}$,
\begin{equation}
\label{eq_goal_concentration_nu}
| \nu(K,\pi,\Qbf,\gamma) | - \pen_n(K,M)
	\leq \eta \Kbf(K,\Qbf,\gamma) + R_n
\end{equation}
for some constant $\eta \in (0,\frac{1}{2})$, some penalty $\pen_n$ and some residual term $R_n$. The above inequality leads to
\begin{align*}
(1 - \eta) \Kbf(\hat{K}, \hat{\Qbf}, \hat{\gamma})
	\leq (1 + \eta) \Kbf(K, \Qbf_{K,M}, \gamma_{K,M}) + 2 \pen_n(K,M) + 2 R_n,
\end{align*}
and the oracle inequality follows by noticing that $\frac{1+\eta}{1-\eta} \leq 1 + 4\eta$ and $\frac{1}{1-\eta} \leq 2$ when $\eta \in (0, \frac{1}{2})$.

Let us now prove equation~\eqref{eq_goal_concentration_nu}.
For all $i \in \Zbb$, $k \in \Nbb^*$, let
\begin{equation}
\label{eq_def_Lik_star}
L_{i,k}^* = \log p^* (Y_i | Y_{i-k}^{i-1}),
\end{equation}
where the process $(Y_t)_{t \geq 1}$ is extended into a process $(Y_t)_{t \in \Zbb}$ by stationarity. Likewise, for all $i \in \Zbb$, $k \in \Nbb^*$, $(K,\pi,\Qbf,\gamma) \in \Sbf_n$ and for all probability distributions $\mu$ on $[K]$, let
\begin{equation*}
L_{i,k,\mu}(K,\Qbf,\gamma) = \log p_{(K,\Qbf,\gamma)} (Y_i | Y_{i-k}^{i-1}, X_{i-k} \sim \mu),
\end{equation*}
where $p_{(K,\Qbf,\gamma)}(\cdot | X_{i-k} \sim \mu)$ is the density of a HMM with parameters $(K, \Qbf, \gamma)$ starting at time $i-k$ with the distribution $\mu$. When $\mu$ is the stationary distribution of the Markov chain under the parameter $(K,\Qbf,\gamma)$, we write $L_{i,k}(K,\Qbf,\gamma)$.
The following remark will be useful in our proofs: since
\begin{align*}
p_{(K,\pi,\Qbf,\gamma)}(X_k = x | Y_1^{k-1}) &= \frac{ \displaystyle
	\sum_{x' \in [K]}  \!\!  p_{(K,\pi,\Qbf,\gamma)}(X_{k-1} = x' | Y_1^{k-2}) \Qbf(x',x) \gamma_{x'}(Y_{k-1})
}{ \displaystyle
	\sum_{x' \in [K]} p_{(K,\pi,\Qbf,\gamma)}(X_{k-1} = x' | Y_1^{k-2}) \gamma_{x'}(Y_{k-1})
} \\
&\in [(\Csigma \log n)^{-1} K^{-1}, \Csigma (\log n) K^{-1}]
\end{align*}
using \ref{Aergodic}, one has for all $k$, $\mu$ and $(K,\pi,\Qbf,\gamma) \in \Sbf_n$
\begin{align}
\label{eq_Li_encadre_par_d}
\left| L_{i,k,\mu}(K,\Qbf,\gamma) - b_\gamma(Y_i) \right| \leq \log (\Csigma \log n).
\end{align}

Assume from now on that $n \geq \exp(\Csigma)$. For all $k, k' \in \Nbb^*$, for all $\mu$, $\mu'$ probability distributions and for all ${(K,\pi,\Qbf,\gamma), (K',\pi',\Qbf',\gamma')} \in \Sbf_n$,
\begin{equation}
\label{eq_Li_borne_par_d}
\hspace{-1em}
\begin{cases}
| L_{i,k,\mu}(K,\Qbf,\gamma) - L_{i,k',\mu'}(K',\Qbf',\gamma') | \leq 4 \log \log n + | b_\gamma(Y_i) | + | b_{\gamma'}(Y_i) |,
\\[5pt]
| L_{i,k,\mu}(K,\Qbf,\gamma) - L_{i,k'}^* | \leq 2 \log \log n + | b_\gamma(Y_i) | + | L_{i,k'}^* |.
\end{cases}
\end{equation}

Let $k \geq 1$ and $D > 0$. Approximate $\nu(K, \pi, \Qbf, \gamma)$ by the deviation
\begin{equation*}
\bar{\nu}_k(t_{(K, \Qbf, \gamma)}^{(D)}) :=
	\frac{1}{n} \sum_{i=1}^n t_{(K, \Qbf, \gamma)}^{(D)}(Y_{i-k}^i) - \Ebb^* [t_{(K, \Qbf, \gamma)}^{(D)}(Y_{-k}^0)]
\end{equation*}
where
\begin{equation*}
t_{(K, \Qbf, \gamma)}^{(D)} : Y_{-k}^0 \longmapsto (L_{0,k}^* - L_{0,k,\mu_t}(K, \Qbf, \gamma)) \one_{|L_{0,k}^*| \vee \left(\sup_{\gamma' \in \Sbf_n^{(\gamma)}} |b_{\gamma'}(Y_0)| \right) \leq D}
\end{equation*}
for a fixed measure $\mu_t$, for instance the uniform measure on $[K]$. Note that $\| t_{(K, \Qbf, \gamma)}^{(D)} \|_\infty \leq 2(D + \log \log n)$ by equation~\eqref{eq_Li_borne_par_d}.

Considering these functions $t_{(K, \Qbf, \gamma)}^{(D)}$ has two advantages. The first one is to limit the time dependency on the past to only $k$ observations, which makes it possible to use the forgetting property of the process $(Y_t)_{t \in \Zbb}$. The second one is to consider bounded functionals of this process, for which Bernstein-like concentration inequalities apply. The error of this approximation is given by the following lemma.
\begin{lemma}
\label{lemma_ecart_barnuk_barnuktronque}
Assume \ref{Atail}, \ref{Aergodic}, \ref{Astar_tail} and \ref{Astar_forgetting} hold. Then there exists $n_0$ depending on $\Csigma$, $C_*, \rho_*, M_\delta$ and $\delta$ such that for all $n \geq n_0$, for all $u \geq 1$, with probability greater than $1 - 2 n e^{-u}$, for all $(K, \pi, \Qbf, \gamma) \in \Sbf_n$,
\begin{equation*}
\left| \nu(K, \pi, \Qbf, \gamma) - \bar{\nu}_k(t_{(K, \Qbf, \gamma)}^{(\CB (\log n) u)}) \right| \leq 10 \CB (\log n) u e^{-u}
	+ \frac{2}{n \rho (1 - \rho)^2} + \frac{4 \rho^{k-1}}{1 - \rho}
\end{equation*}
where $\rho = 1 - (\Csigma \log n)^{-2}$. In particular, if $k \geq \Csigma^2 (\log n)^3$ and $n \geq n_0 \vee \sqrt{30 \CB}$, for all $D \geq 3 \CB (\log n)^2$, with probability greater than $1 - 2 n^{-2}$,
\begin{equation}
\label{eq_ecart_nu_nuk_general}
\left| \nu(K, \pi, \Qbf, \gamma) - \bar{\nu}_k(t_{(K, \Qbf, \gamma)}^{(D)}) \right| \leq 13 \Csigma^4 \frac{(\log n)^4}{n}.
\end{equation}
\end{lemma}

\begin{proof}
Proof in Section~\ref{sec_preuve_lemma_troncature}.
\end{proof}

The following theorem is our main technical result. It shows that $\bar{\nu}_k(t_{(K, \Qbf, \gamma)}^{(D)})$ can be controlled uniformly on all models with high probability.
\begin{theorem}
\label{th_penalite_L2}
Assume \ref{Aergodic}, \ref{Aentropy} and \ref{Astar_mixing}. 
Also assume that $D \geq \log n$, that $k \geq n_* + 1$ and that there exists $n_1$ such that for all $n \geq n_1$, for all $K \leq n$ and $M$ such that $m_M \leq n$,
\begin{equation}
\label{eq_condition_fKM}
   14 \pi \, ( m_M K + K^2 - 1)   e^{-4D} (\log n)^2 (k + \log \Caux(M,K,D,n)) \leq n.
\end{equation}

Let $(w_M)_{M \in \Mcal}$ be a sequence of positive numbers such that $\sum_M e^{-w_M} \leq e-1$. Then there exist constants $C_\pen$ and $A$ depending on $n_*$ and $c_*$ and a numerical constant $n_0$ such that for all $\epsilon > 0$ and $n \geq n_1 \vee n_0$, the following holds.

Let $\pen_n$ be a function such that for all $K \leq n$ and $M$ such that $m_M \leq n$,
\begin{multline}
\label{eq_condition_penalite_generale}
\pen_n(K,M) \geq  \frac{C_\pen}{n} k^2 \left( \frac{1}{\epsilon} \vee \frac{D (\log n)^2}{k} \right) \times \\
		\Big( w_M + ( m_M K + K^2 - 1) D (\log n)^2 (D + k \log \log n + \log \Caux) \Big).
\end{multline}
Then for all $s > 0$, with probability larger than $1 - e^{-s}$, for all $K \leq n$ and $M$ such that $m_M \leq n$ and for all $(K, \pi, \Qbf, \gamma) \in S_{K,M,n}$,
\begin{multline}
\label{eq_control_nukmoinspen_generale}
| \bar{\nu}_k(t_{(K, \Qbf, \gamma)}^{(D)}) | - \pen_n(K,M) \leq \epsilon \Ebb[t_{(K, \Qbf, \gamma)}^{(D)}(Y_{-k}^0)^2] \\
	+ A k^2 \left(\frac{1}{\epsilon} \vee \frac{D (\log n)^2}{k} \right) \frac{s}{n}.
\end{multline}
\end{theorem}

\begin{proof}
Proof in Section~\ref{sec_proof_control_barnuk}.
\end{proof}

The last step is to control the variance term $\Ebb[t_{(K, \Qbf, \gamma)}^{(D)}(Y_{-k}^0)^2]$ by $\Kbf(K, \Qbf, \gamma)$.
\begin{lemma}
\label{lemma_control_esperance_quadratique_par_Kbf}
Assume \ref{Atail}, \ref{Aergodic}, \ref{Astar_tail} and \ref{Astar_forgetting} hold.
There exists a constant $n_0$ depending on $M_\delta$, $\delta$, $\rho_*$, $C_*$ and $\Csigma$ such that for all $n \geq n_0$, $k \geq \Csigma (\log n)^3$, $D > 0$ and $(K, \pi, \Qbf, \gamma) \in \Sbf_n$,
\begin{equation*}
\frac{1}{44 \CB^2 (\log n)^4} \Ebb^* [ t_{(K, \Qbf, \gamma)}^{(D)}(Y_{i-k}^i)^2 ]
	\leq \Kbf(K, \Qbf, \gamma)
		+ \frac{22}{n}.
\end{equation*}
\end{lemma}

\begin{proof}
Proof in Section~\ref{sec_proof_controle_dvar}.
\end{proof}

Take $u = 3 \log n$ in order to have $n e^{-u} \leq n^{-2}$ in Lemma~\ref{lemma_ecart_barnuk_barnuktronque}. Note that $u \geq 1$ for all $n \geq e$. Based on Lemma~\ref{lemma_ecart_barnuk_barnuktronque} and~\ref{lemma_control_esperance_quadratique_par_Kbf}, also take
\begin{equation*}
\begin{cases}
D = \CB (\log n) u = 3 \CB (\log n)^{2}, \\
k = \Csigma^2 (\log n)^3.
\end{cases}
\end{equation*}

In the following, we assume $n \geq e \vee \exp([(n_*+1) / \Csigma^2]^{1/3})$, so that $k \geq n_* + 1$ and $D \geq \log n$.
Let $\eta \leq 1$. In order to get $\epsilon \Ebb^*[t_{(K, \Qbf, \gamma)}^{(D)}(Y_{i-k}^i)^2] \leq \eta \Kbf(K, \Qbf, \gamma) + \frac{22 \eta}{n}$ using Lemma~\ref{lemma_control_esperance_quadratique_par_Kbf}, take
\begin{equation*}
\frac{1}{\epsilon} = \frac{1}{\eta} 44 \CB^2 (\log n)^4.
\end{equation*}

When assumption \ref{Agrowth} holds and $m_M \leq n$ and $K \leq n$, equation~\eqref{eq_condition_fKM} is implied by
\begin{equation*}
28 \pi n^2 (\log n)^2 e^{-12 \CB (\log n)^2} (\Csigma^2 (\log n)^3 \log \log n + n^{\zeta}) \leq n
\end{equation*}
for all $n \geq n_\text{growth}$, which is true for $n \geq n_1$ for a constant $n_1$ depending only on $n_\text{growth}$, $\Csigma$ and $\zeta$.

Moreover, there exists a constant $C_\epsilon$ depending only on $\Csigma$ and $\CB$ such that for all $n$,
\begin{equation*}
\frac{1}{\epsilon} \vee \frac{D (\log n)^2}{k}
	\leq \frac{C_\epsilon}{\eta} (\log n)^4.
\end{equation*}

Thus, there exists an integer $n_0''$ depending on $\Csigma$ and $\CB$ (for instance $\exp(3\CB / \Csigma^2)$) such that for all $n \geq n_0''$ 
equation~\eqref{eq_condition_penalite_generale} is implied by
\begin{multline*}
\pen_n(K,M) \geq \frac{C_\pen}{n} \Csigma^4 (\log n)^6 \frac{C_\epsilon}{\eta} (\log n)^{4} \\
	\times \Bigg[ w_M + 6 \CB (\log n)^{4} ( m_M K + K^2 - 1) \\
		\times \left( \Csigma^2 (\log n)^3 \log \log n
		+ \log \Caux(M,K,3\CB(\log n)^2,n) \right)\Bigg],
\end{multline*}
so if in addition $n$ is larger than the thresholds of Theorem~\ref{th_penalite_L2} and Lemma~\ref{lemma_control_esperance_quadratique_par_Kbf}, equation~\eqref{eq_control_nukmoinspen_generale} and Lemma~\ref{lemma_control_esperance_quadratique_par_Kbf} imply for all $s > 0$, with probability at least $1-e^{-s}$, for all $K \leq n$ and $M$ such that $m_M \leq n$ and all $(K,\pi,\Qbf,\gamma) \in S_{K,M,n}$,
\begin{align}\nonumber
\label{eq_control_nubar_final}
| \bar{\nu}_k(t_{(K, \Qbf, \gamma)}^{(D)}) | -& \pen_n(K,M) \\ \nonumber
	&\leq \eta \Kbf(K, \Qbf, \gamma)
		+ \frac{22 \eta}{n} + A \Csigma^4 (\log n)^6 \frac{C_\epsilon}{\eta} (\log n)^{4} \frac{s}{n} \\
	&\leq \eta \Kbf(K, \Qbf, \gamma)
	+ 23 A \Csigma^4 (\log n)^6 \frac{C_\epsilon}{\eta} (\log n)^{4} \frac{s}{n}
\end{align}
since we may assume $A \geq 1$ without loss of generality. Therefore, putting together equations~\eqref{eq_ecart_nu_nuk_general} and~\eqref{eq_control_nubar_final} shows
\begin{align*}
| \nu(K,\pi,\Qbf,\gamma) | -& \pen_n(K,M) \\
	&\leq \eta \Kbf(K,\Qbf,\gamma) + \frac{23 A \Csigma^4 C_\epsilon}{\eta} s \frac{(\log n)^{10}}{n} + 13 \Csigma^4 \frac{(\log n)^4}{n} \\
	&\leq \eta \Kbf(K,\Qbf,\gamma) + \frac{36 A \Csigma^4 C_\epsilon}{\eta} s \frac{(\log n)^{10}}{n}
\end{align*}
which is equation~\eqref{eq_goal_concentration_nu} with the appropriate residual terms for Theorem~\ref{th_oracle_simplifie}.

\subsection{Proofs}
\label{sec_preuves_techniques_oracle}

Let us first state two lemmas that will be of use in subsequent proofs.
\begin{lemma}
\label{lem_tail_moments}
Assume \ref{Atail} and \ref{Astar_tail}. Then there exists a constant $n_0$ depending on $\delta$ and $M_\delta$ such that for all $n \geq n_0$, for all $i,k$, and for all $u \geq 1$,
\begin{equation*}
\Pbb^*[ |L_{i,k}^*| \geq \CB (\log n) u ] \leq e^{-u}
\end{equation*}
where $L_{i,k}^* = \log p^* (Y_i | Y_{i-k}^{i-1})$ as defined in~\eqref{eq_def_Lik_star}, and writing $D = \CB (\log n) u$,
\begin{gather*}
\Ebb\left[ \sup_{\gamma \in \Sbf_n^{(\gamma)}} \left| b_\gamma(Y_1) \right| \one_{\sup_{\gamma \in \Sbf_n^{(\gamma)}} \left| b_\gamma(Y_1) \right| \geq D} \right]
	\vee \Ebb\left[ |L_{i,k}^*| \one_{|L_{i,k}^*| \geq D} \right]
	\leq 2 D e^{-u}, \\
\Ebb\left[ \sup_{\gamma \in \Sbf_n^{(\gamma)}} \left| b_\gamma(Y_1) \right|^2 \one_{\sup_{\gamma \in \Sbf_n^{(\gamma)}} \left| b_\gamma(Y_1) \right| \geq D} \right]
	\vee \Ebb\left[ |L_{i,k}^*|^2 \one_{|L_{i,k}^*| \geq D} \right]
	\leq 5 D^2 e^{-u}.
\end{gather*}
\end{lemma}

\begin{proof}
Let $i \in \Zbb$, $k \in \Nbb$ and $v > 0$. By \ref{Astar_tail} and Markov's inequality,
\begin{align*}
\Pbb^* \left[ L_{i,k}^* \geq v \right]
	&= \Pbb^* \left[ p^*(Y_i | Y_{i-k}^{i-1}) \geq e^v \right] \\
	&\leq e^{- \delta v} \Ebb^* \left[ (p^*(Y_i | Y_{i-k}^{i-1}))^\delta \right] \\
	&\leq e^{\log M_\delta - \delta v}.
\end{align*}

On the other hand,
\begin{align*}
\Pbb^* \left[ L_{i,k}^* \leq -v \right]
	&= \Pbb^* \left[ p^*(Y_i | Y_{i-k}^{i-1}) \leq e^{-v} \right] \\
	&= \Ebb^* \left[ \int \one_{p^*(y | Y_{i-k}^{i-1}) \leq e^{-v}} p^*(y | Y_{i-k}^{i-1}) \lambda(dy) \right] \\
	&\leq e^{-v}.
\end{align*}

Thus, there exists $B^* \geq 1$ such that if $u \geq 1$, $\Pbb^*[ |L_{i,k}^*| \geq B^* u ] \leq e^{-u}$. Therefore, for all $n \geq \exp(B^*)$, the first equation holds, and under \ref{Atail}, the variables $|L_{i,k}^*|$ and $\sup_{\gamma \in \Sbf_n^{(\gamma)}} \left| b_\gamma(Y_1) \right|$ are dominated by $\CB (\log n) (W \vee 1)$ where $W$ is an exponential random variable with parameter $1$. To conclude, note that for all $u > 0$,
\begin{gather*}
\Ebb^*[ W \one_{W \geq u} ] \leq (1 + u) e^{-u} \leq 2ue^{-u}, \\
\Ebb^*[ W^2 \one_{W \geq u} ] \leq (u^2 + 2u + 2) e^{-u} \leq 5u^2e^{-u}.
\end{gather*}
\end{proof}

\begin{lemma}
\label{lemma_ecart_Likx}
Assume \ref{Aergodic} and \ref{Astar_forgetting}.
\begin{enumerate}
\item Let $\rho = 1 - (\Csigma \log n)^{-2}$. Then for all $i$, $k$, $k'$, $\mu$ and $\mu'$,
\begin{equation*}
\sup_{(K,\pi,\Qbf,\gamma) \in \Sbf_n}
	| L_{i,k,\mu}(K,\Qbf,\gamma) - L_{i,k',\mu'}(K,\Qbf,\gamma) |
		\leq \rho^{k \wedge k' - 1} / (1 - \rho)
\end{equation*}
and there exists a process $(L_{i, \infty})_{i \in \Zbb}$ such that for all $i$, $k$ and $\mu$,
\begin{equation*}
\sup_{(K,\pi,\Qbf,\gamma) \in \Sbf_n} | L_{i,k,\mu}(K,\Qbf,\gamma) - L_{i,\infty}(K,\Qbf,\gamma) | \leq \rho^{k-1} / (1 - \rho).
\end{equation*}

\item For all $i$, $k$ and $k'$,
$| L_{i,k}^* - L_{i,k'}^* | \leq C_* \rho_*^{k \wedge k' - 1}$ and there exists a process $(L_{i, \infty}^*)_{i \in \Zbb}$ such that for all $i$ and $k$,
\begin{equation*}
| L_{i,k}^* - L_{i,\infty}^* | \leq C_* \rho_*^{k-1}.
\end{equation*}

\item \label{point3} Under $\Pbb^*$, the processes $(L_{i,\infty}^*)_{i \in \Zbb}$ and $(L_{i,\infty}(K,\Qbf,\gamma))_{i \in \Zbb}$ are stationary for all ${(K,\pi,\Qbf,\gamma) \in \Sbf_n}$. Assume \ref{Astar_mixing}, \ref{Atail} and \ref{Astar_tail}, then they are also ergodic, integrable and
\begin{equation*}
l(K,\Qbf,\gamma) = \Ebb^* [ L_{1,\infty}(K,\Qbf,\gamma) ] \quad \text{and} \quad
l^* = \Ebb^* [ L_{1,\infty}^* ].
\end{equation*}
\end{enumerate}
\end{lemma}

\begin{proof}
The first point is a result from \cite{douc2004asymptotic}.

The second point follows directly from \ref{Astar_forgetting}.

The third point follows from the ergodicity of $(Y_t)_{t \geq 1}$ under \ref{Astar_mixing}, from the integrability of $L_{i,\infty}$ and $L_{i,\infty}^*$ under \ref{Atail} and \ref{Astar_tail} by Lemma~\ref{lem_tail_moments} and from Lemmas \ref{lemma_Barron} and \ref{lemma_Leroux} for the definition of $l$ and $l^*$.
\end{proof}

\subsubsection{Proof of Lemma~\ref{lemma_ecart_barnuk_barnuktronque}}
\label{sec_preuve_lemma_troncature}

Let ${t_{(K,\Qbf,\gamma)} : Y_{-k}^0 \longmapsto L_{0,k}^* - L_{0,k,\mu_t}(K,\Qbf,\gamma)}$. Then
\begin{align*}
\nu(K,\pi,\Qbf,\gamma) &- \bar{\nu}_k(t_{(K,\Qbf,\gamma)}) \\
	=& \; \frac{1}{n} \sum_{i=1}^n (L_{i,i-1}^* - L_{i,k}^*)
		- \frac{1}{n} \sum_{i=1}^n (L_{i,i-1,\pi}(K,\Qbf,\gamma) - L_{i,k,\mu_t}(K,\Qbf,\gamma)) \\
		&- \Ebb[L_{0,\infty}^* - L_{0,k}^*] + \Ebb[L_{0,\infty}(K,\Qbf,\gamma) - L_{0,k,\mu_t}(K,\Qbf,\gamma)].
\end{align*}

Thus, by Lemma~\ref{lemma_ecart_Likx},
\begin{align*}
| \nu(K,\pi,\Qbf,\gamma) &- \bar{\nu}_k(t_{(K,\Qbf,\gamma)}) | \\
	&\leq \frac{1}{n} \sum_{i=1}^n \frac{\rho^{(i-1) \wedge k - 1}}{1 - \rho}
		+ C_* \frac{1}{n} \sum_{i=1}^n \rho_*^{(i-1) \wedge k - 1}
		+ \frac{\rho^{k - 1}}{1 - \rho}
		+ C_* \rho_*^{k - 1} \\
	&\leq \frac{1}{n \rho (1 - \rho)^2} + \frac{2 \rho^{k-1}}{1 - \rho}
		+ C_* \left( \frac{1}{n \rho_* (1 - \rho_*)} + 2 \rho_*^{k-1} \right) \\
	&\leq \frac{2}{n \rho (1 - \rho)^2} + \frac{4 \rho^{k-1}}{1 - \rho}
\end{align*}
as soon as $\rho_* \leq \rho$ and $C_* \leq 1/(1 - \rho)$, which holds when $\Csigma \log n \geq (C_* \vee (1-\rho_*)^{-1})^{1/2}$, in particular when $\log n \geq (C_* \vee (1-\rho_*)^{-1})^{1/2}$.

Let $u \geq 1$ and $D = \CB (\log n) u$ and assume that $n \geq n_0$ from Lemma~\ref{lem_tail_moments}. Then
\begin{align*}
\bar{\nu}_k(t_{(K,\Qbf,\gamma)}) - \bar{\nu}_k(t_{(K,\Qbf,\gamma)}^{(D)})
	 ={} & \frac{1}{n} \sum_{i=1}^n t_{(K,\Qbf,\gamma)}(Y_{i-k}^i) \one_{|L_{i,k}^*| \vee (\sup_{\gamma' \in \Sbf_n^{(\gamma)}} | b_{\gamma'}(Y_i) |) > D} \\
	 &- \Ebb^* [t_{(K,\Qbf,\gamma)}(Y_{-k}^0) \one_{|L_{0,k}^*| \vee (\sup_{\gamma' \in \Sbf_n^{(\gamma)}} | b_{\gamma'}(Y_0) |) > D}].
\end{align*}

We restrict ourselves to the event $\bigcap_{i=1}^n \{ |L_{i,k}^*| \vee (\sup_{\gamma' \in \Sbf_n^{(\gamma)}} | b_{\gamma'}(Y_i) |) \leq D \}$, which occurs with probability greater than $1 - 2 n e^{-u}$ using assumption \ref{Atail} and Lemma~\ref{lem_tail_moments}. On this event,
\begin{equation*}
\frac{1}{n} \sum_{i=1}^n t_{(K,\Qbf,\gamma)}(Y_{i-k}^i) \one_{|L_{i,k}^*| \vee (\sup_{\gamma' \in \Sbf_n^{(\gamma)}} \left| b_{\gamma'}(Y_i) \right|) > D} = 0.
\end{equation*}

Moreover,
\begin{multline*}
| \Ebb^*[t_{(K,\Qbf,\gamma)}(Y_{-k}^0) - t_{(K,\Qbf,\gamma)}^{(D)}(Y_{-k}^0)] | \\
	= \Ebb^*[|t_{(K,\Qbf,\gamma)}(Y_{-k}^0)| \one_{|L_{0,k}^*| \vee (\sup_{\gamma' \in \Sbf_n^{(\gamma)}} |b_{\gamma'}(Y_0)|) > D}].
\end{multline*}

Equation~\eqref{eq_Li_borne_par_d} ensures that $|t_{(K,\Qbf,\gamma)}(Y_{-k}^0)| \leq |L_{0,k}^*| + \sup_{\gamma' \in \Sbf_n^{(\gamma)}} |b_{\gamma'}(Y_0)| + 2 \log \log n$ when $n \geq \exp(\Csigma)$, so that
\begin{align*}
| \Ebb^* & [t_{(K,\Qbf,\gamma)}(Y_{-k}^0) -t_{(K,\Qbf,\gamma)}^{(D)}(Y_{-k}^0)] |  \\
	\leq{} & \; \Ebb^*\left[|L_{0,k}^*| \left(\one_{|L_{0,k}^*| > D} + \one_{|L_{0,k}^*| \leq D < \underset{\gamma' \in \Sbf_n^{(\gamma)}}{\sup} |b_{\gamma'}(Y_0)|}\right)\right] \\
		&+ \Ebb^*\left[\sup_{\gamma' \in \Sbf_n^{(\gamma)}} |b_{\gamma'}(Y_0)| \left( \one_{\underset{\gamma' \in \Sbf_n^{(\gamma)}}{\sup} |b_{\gamma'}(Y_0)| > D}
				+ \one_{\underset{\gamma' \in \Sbf_n^{(\gamma)}}{\sup} |b_{\gamma'}(Y_0)| \leq D < |L_{0,k}^*|} \right)\right] \\
		&+ 2 (\log \log n) \Ebb^*\left[\left( \one_{|L_{0,k}^*| > D} + \one_{\underset{\gamma' \in \Sbf_n^{(\gamma)}}{\sup} |b_{\gamma'}(Y_0)| > D} \right) \right].
\end{align*}

Thus, by Lemma~\ref{lem_tail_moments},
\begin{align*}
| \Ebb^* [t_{(K,\Qbf,\gamma)} & (Y_{-k}^0) - t_{(K,\Qbf,\gamma)}^{(D)}(Y_{-k}^0)] | \\
	\leq{} & 2De^{-u} + D \Pbb^*\left(\sup_{\gamma' \in \Sbf_n^{(\gamma)}} |b_{\gamma'}(Y_0)| > D\right)
		+ 2De^{-u} + D \Pbb^*(|L_{0,k}^*| > D) \\
		&+ 2 (\log \log n) \left( \Pbb^*\left(\sup_{\gamma' \in \Sbf_n^{(\gamma)}} |b_{\gamma'}(Y_0)| > D\right) + \Pbb^*(|L_{0,k}^*| > D) \right) \\
	\leq{} & 6D e^{-u} + 4 (\log \log n) e^{-u}.
\end{align*}

Finally, using $\log \log n \leq \log n \leq \CB \log n u$ when $u \geq 1$ concludes the proof of the first equation.

For the second equation, take $u = 3 \log n$. Since $\rho^{k} \leq n^{-1}$ when $k \geq \Csigma^2 (\log n)^3$ and $u \longmapsto u e^{-u}$ is nonincreasing on $[1, +\infty)$, for all $D \geq 3 \CB (\log n)^2$,
\begin{align*}
\Big| \nu(K, \pi, \Qbf, \gamma) &- \bar{\nu}_k(t_{(K, \Qbf, \gamma)}^{(D)}) \Big| \\
	&\leq 10 \CB (\log n) (3 \log n) e^{- 3 \log n}
		+ \frac{2 (\Csigma \log n)^4}{n \rho} + \frac{4 (\Csigma \log n)^2}{n \rho} \\
	&\leq 30 \CB \frac{(\log n)^2}{n^3}
		+ \frac{12 (\Csigma \log n)^4}{n}
	\leq \frac{13 (\Csigma \log n)^4}{n}
\end{align*}
for $n \geq \sqrt{30 \CB}$ (using $\rho \leq 1/2$ for the second line).

\subsubsection{Proof of Lemma~\ref{lemma_control_esperance_quadratique_par_Kbf}}
\label{sec_proof_controle_dvar}

\begin{lemma}
\label{lemma_dvar}
Assume \ref{Atail}, \ref{Aergodic} and \ref{Astar_tail} hold. Let
\begin{equation*}
\Vbf(K,\Qbf,\gamma) := \Ebb^* \left[
	(L_{0,\infty}^* - L_{0,\infty}(K,\Qbf,\gamma))^2
\right].
\end{equation*}
Then for all $n \geq e^4 \vee \exp(\Csigma)$,
\begin{equation*}
\frac{1}{44 \CB^2 (\log n)^4} \Vbf(K,\Qbf,\gamma)
	\leq \Kbf(K,\Qbf,\gamma) + \frac{11}{n}.
\end{equation*}
\end{lemma}

\begin{proof}
We need the following lemma:
\begin{lemma}[\cite{STG13bayesiandensity}, Lemma 4]
\label{lemma_STG}
For any two probability measures $P$ and $Q$ with density $p$ and $q$ and any $\lambda \in (0, e^{-4}]$,
\begin{equation*}
\Ebb_P \left(\log\frac{p}{q}\right)^2 \leq H(P,Q)^2 \left(12 + 2 \left(\log\frac{1}{\lambda}\right)^2 \right)
	+ 8 \Ebb_P\left[ \left(\log\frac{p}{q}\right)^2 \one\left(\frac{p}{q} \geq \frac{1}{\lambda}\right) \right] 
\end{equation*}
where $H(P,Q)$ is the Hellinger distance between $P$ and $Q$:
\begin{equation*}
H(P,Q)^2 = -2 \Ebb_P [(q/p)^{1/2} - 1] = \int (\sqrt{p} - \sqrt{q})^2 d \lambda.
\end{equation*}
\end{lemma}

\noindent
Let $n \in \Nbb^*$ and $D' = \CB (\log n)^2$.
Take $P = \Pbb^*_{Y_0 | Y_{-\infty}^{-1}}$ and $Q = \Pbb_{Y_0 | Y_{-\infty}^{-1}, (K,\Qbf,\gamma)}$, so that $\Ebb_P ( \log\frac{p}{q} )^2 = \Vbf(K,\Qbf,\gamma)$. Using equation~\eqref{eq_Li_borne_par_d} for $n \geq \exp(\Csigma)$,
\begin{align*}
\left(\log\frac{p}{q}\right)^2 &\leq \left( \sup_{\gamma' \in \Sbf_n^{(\gamma)}} | b_{\gamma'}(Y_0) | + |L_{0,\infty}^*|
		+ 2 \log \log n \right)^2 \\
	&\leq 3 \sup_{\gamma' \in \Sbf_n^{(\gamma)}} | b_{\gamma'}(Y_0) |^2
		+ 3 |L_{0,\infty}^*|^2
		+ 12 (\log \log n)^2
\end{align*}

Let $\lambda > 0$ be such that $2D' = \log \frac{1}{\lambda} - 2 \log \log n$. Note that $\lambda \leq e^{-4}$ when $n \geq e^4$. By equation~\eqref{eq_Li_borne_par_d},
\begin{align*}
\one\left(\frac{p}{q} \geq \frac{1}{\lambda}\right)
	&\leq \one\left(\sup_{\gamma' \in \Sbf_n^{(\gamma)}} | b_{\gamma'}(Y_0) | + |L_{0,\infty}^*| \geq \log \frac{1}{\lambda} - 2 \log \log n \right) \\
	&\leq \one\left(\sup_{\gamma' \in \Sbf_n^{(\gamma)}} | b_{\gamma'}(Y_0) | \vee |L_{0,\infty}^*| \geq D' \right),
\end{align*}
hence
\begin{align*}
&8 \Ebb_P \left[ \left(\log\frac{p}{q}\right)^2 \one\left(\frac{p}{q} \geq \frac{1}{\lambda}\right) \right] \\
	&\leq 24 \Ebb^*\left[|L_{0,\infty}^*|^2 \left(\one_{|L_{0,\infty}^*| > D'} + \one_{|L_{0,\infty}^*| \leq D' < \underset{\gamma' \in \Sbf_n^{(\gamma)}}{\sup} |b_{\gamma'}(Y_0)|}\right)\right] \\
		&\quad + 24 \Ebb^* \left[\sup_{\gamma' \in \Sbf_n^{(\gamma)}} |b_{\gamma'}(Y_0)|^2 \left( \one_{\underset{\gamma' \in \Sbf_n^{(\gamma)}}{\sup} |b_{\gamma'}(Y_0)| > D'}
		+ \one_{\underset{\gamma' \in \Sbf_n^{(\gamma)}}{\sup} |b_{\gamma'}(Y_0)| \leq D' < |L_{0,\infty}^*|} \right)\right] \\
		&\quad + 96 (\log \log n)^2 \Ebb^*\left[\one_{|L_{0,\infty}^*| > D'} + \one_{\underset{\gamma' \in \Sbf_n^{(\gamma)}}{\sup} |b_{\gamma'}(Y_0)| > D'} \right],
\end{align*}
and by Lemma~\ref{lem_tail_moments} (for $n$ large enough)
\begin{align*}
8 \Ebb_P \left[ \left(\log\frac{p}{q}\right)^2 \one\left(\frac{p}{q} \geq \frac{1}{\lambda}\right) \right]
	&\leq 24 \left(\frac{5D'^2}{n} + \frac{D'^2}{n} \right) + 24 \left(\frac{5D'^2}{n} + \frac{D'^2}{n} \right) \\
		&\qquad+ 96 \frac{2(\log \log n)^2}{n} \\
	&\leq \frac{480}{n} D'^2
\end{align*}
using $\log \log n \leq D'$ for $n \geq e$. Therefore, by Lemma~\ref{lemma_STG},
\begin{align*}
\Vbf(K,\Qbf,\gamma) &\leq \Ebb^*_{Y_{-\infty}^{-1}}\left[ H(\Pbb^*_{Y_0 | Y_{-\infty}^{-1}}, \Pbb_{Y_0 | Y_{-\infty}^{-1}, (K,\Qbf,\gamma)})^2 \right] (12 + 2 ( 2D' + 2 \log \log n )^2 ) \\
	&\qquad + \frac{480}{n} D'^2 \\
	&\leq \Ebb^*_{Y_{-\infty}^{-1}} \left[ KL(\Pbb^*_{Y_0 | Y_{-\infty}^{-1}} \| \Pbb_{Y_0 | Y_{-\infty}^{-1}, (K,\Qbf,\gamma)}) \right] (12 + 32 D'^2 ) + \frac{480}{n} D'^2
\end{align*}
using that the Kullback Leibler divergence is lower bounded by the Hellinger distance. Finally, since
$\Ebb^*_{Y_{-\infty}^{-1}}[ KL(\Pbb^*_{Y_0 | Y_{-\infty}^{-1}} \| \Pbb_{Y_0 | Y_{-\infty}^{-1}, (K,\Qbf,\gamma)})] = \Kbf(K,\Qbf,\gamma)$,
\begin{align*}
\Vbf(K,\Qbf,\gamma)
	&\leq 44 D'^2 \Kbf(K,\Qbf,\gamma) + \frac{480}{n} D'^2.
\end{align*}
\end{proof}

Next, let $D > 0$ and let us bound the difference between $\Vbf(K,\Qbf,\gamma)$ and $\Ebb^* [ t_{(K,\Qbf,\gamma)}^{(D)}(Y_{i-k}^i)^2 ]$. Taking ${t_{(K,\Qbf,\gamma)} : Y_{-k}^0 \longmapsto L_{0,k}^* - L_{0,k,x}(K,\Qbf,\gamma)}$, by definition of $t_{(K,\Qbf,\gamma)}^{(D)}$
\begin{equation*}
\Ebb^* [ t_{(K,\Qbf,\gamma)}^{(D)}(Y_{i-k}^i)^2 ] \leq \Ebb^* [ t_{(K,\Qbf,\gamma)}(Y_{i-k}^i)^2 ].
\end{equation*}

Then,
\begin{align*}
| \Ebb^* &[ t_{(K,\Qbf,\gamma)}(Y_{i-k}^i)^2 ] - \Vbf(K,\Qbf,\gamma) | \\
	={} & \left| \Ebb^* \left[ (L_{0,k}^* - L_{0,k,x}(K,\Qbf,\gamma))^2 \right]
		- \Ebb^* \left[ (L_{0,\infty}^* - L_{0,\infty}(K,\Qbf,\gamma))^2 \right] \right| \\
	\leq{} & \Ebb^* | ( (L_{0,k}^* - L_{0,\infty}^*) - (L_{0,k,x} - L_{0,\infty})(K,\Qbf,\gamma) ) \\
			&\quad \times ( (L_{0,k}^* - L_{0,k,x}(K,\Qbf,\gamma)) + (L_{0,\infty}^* - L_{0,\infty}(K,\Qbf,\gamma)) ) | \\
	\leq{} & 2 \frac{\rho^{k-1}}{1 - \rho} \left( \Ebb^* \left[ 2 \sup_{\gamma' \in \Sbf_n^{(\gamma)}}|b_{\gamma'}(Y_0)| + |L_{0,k}^*| + |L_{0,\infty}^*| \right] + 4 \log  \log n \right)
\end{align*}
by Lemma~\ref{lemma_ecart_Likx} and equation~\eqref{eq_Li_borne_par_d}, provided $\rho_* \leq \rho$ and $C_* \leq 1/(1 - \rho)$ (which is ensured by $\log n \geq (C_* \vee (1-\rho_*)^{-1})^{1/2}$). Note that the condition $k \geq \Csigma^2 (\log n)^3$ ensures that $\rho^k \leq n^{-1}$, and that $\rho \leq 1/2$ when $n \geq e^4$.
The expectation can be upper bounded using Lemma~\ref{lem_tail_moments} with $u=1$:
\begin{align*}
| \Ebb^* [ t_{(K,\Qbf,\gamma)}(Y_{i-k}^i)^2 ] - \Vbf(K,\Qbf,\gamma) |
	&\leq \frac{2}{n \rho(1 - \rho)} ( 8 \CB \log n + 4 \log \log n ) \\
	&\leq \frac{48 \Csigma^2 \CB}{n} (\log n)^3.
\end{align*}

Therefore, under the assumptions of Lemma~\ref{lemma_dvar}, if $D \geq \CB (\log n)^2$,
\begin{align*}
\frac{\Ebb^* [ t_{(K,\Qbf,\gamma)}^{(D)}(Y_{i-k}^i)^2 ]}{44 \CB^2 (\log n)^4}
	&\leq \Kbf(K,\Qbf,\gamma)
		+ \frac{11}{n}
		+ \frac{1}{44 \CB^2 (\log n)^4} \frac{48 \Csigma^2 \CB}{n} (\log n)^3 \\
	&\leq \Kbf(K,\Qbf,\gamma)
		+ \frac{11}{n}
		+ \frac{48 \Csigma^2}{44 n \log n} \\
	&\leq \Kbf(K,\Qbf,\gamma)
		+ \frac{22}{n}
\end{align*}
for $n$ larger than a constant that only depends on $\Csigma$, which concludes the proof.

\section*{Acknowledgements}

I am grateful to \'Elisabeth Gassiat for her precious advice and insightful discussions. I would also like to thank the anonymous referee for his patience and very helpful review.

\bibliography{these}

\begin{thebibliography}{37}
\providecommand{\natexlab}[1]{#1}
\providecommand{\url}[1]{\texttt{#1}}
\expandafter\ifx\csname urlstyle\endcsname\relax
  \providecommand{\doi}[1]{doi: #1}\else
  \providecommand{\doi}{doi: \begingroup \urlstyle{rm}\Url}\fi

\bibitem[Alexandrovich et~al.(2016)Alexandrovich, Holzmann, and Leister]{AH14}
Grigory Alexandrovich, Hajo Holzmann, and Anna Leister.
\newblock Nonparametric identification and maximum likelihood estimation for
  hidden {M}arkov models.
\newblock \emph{Biometrika}, 103\penalty0 (2):\penalty0 423--434, 2016.

\bibitem[Anandkumar et~al.(2012)Anandkumar, Hsu, and Kakade]{AHK12}
Animashree Anandkumar, Daniel~J Hsu, and Sham~M Kakade.
\newblock A method of moments for mixture models and hidden {M}arkov models.
\newblock In \emph{COLT}, volume~1, page~4, 2012.

\bibitem[Arlot and Celisse(2010)]{arlot2010survey}
Sylvain Arlot and Alain Celisse.
\newblock A survey of cross-validation procedures for model selection.
\newblock \emph{Statistics surveys}, 4:\penalty0 40--79, 2010.

\bibitem[Barron(1985)]{barron1985entropyrate}
Andrew~R Barron.
\newblock The strong ergodic theorem for densities: generalized
  {S}hannon-{M}c{M}illan-{B}reiman theorem.
\newblock \emph{The Annals of Probability}, 13\penalty0 (4):\penalty0
  1292--1303, 1985.

\bibitem[Baum and Petrie(1966)]{baum1966HMM}
Leonard~E Baum and Ted Petrie.
\newblock Statistical inference for probabilistic functions of finite state
  {M}arkov chains.
\newblock \emph{The Annals of Mathematical Statistics}, 37\penalty0
  (6):\penalty0 1554--1563, 1966.

\bibitem[Bonhomme et~al.(2016)Bonhomme, Jochmans, and
  Robin]{bonhomme2016multiview}
St{\'e}phane Bonhomme, Koen Jochmans, and Jean-Marc Robin.
\newblock Non-parametric estimation of finite mixtures from repeated
  measurements.
\newblock \emph{Journal of the Royal Statistical Society: Series B (Statistical
  Methodology)}, 78\penalty0 (1):\penalty0 211--229, 2016.

\bibitem[Boyd et~al.(2014)Boyd, Punt, Weimerskirch, and
  Bertrand]{boyd2014HmmUsesOfBirdTrajectories}
Charlotte Boyd, Andr{\'e}~E Punt, Henri Weimerskirch, and Sophie Bertrand.
\newblock Movement models provide insights into variation in the foraging
  effort of central place foragers.
\newblock \emph{Ecological modelling}, 286:\penalty0 13--25, 2014.

\bibitem[Bradley(2005)]{bradley2005strongmixingsurvey}
Richard~C Bradley.
\newblock Basic properties of strong mixing conditions. {A} survey and some
  open questions.
\newblock \emph{Probability surveys}, 2:\penalty0 107--144, 2005.

\bibitem[Couvreur and Couvreur(2000)]{CC00}
Laurent Couvreur and Christophe Couvreur.
\newblock Wavelet-based non-parametric {HMM}'s: theory and applications.
\newblock In \emph{Acoustics, Speech, and Signal Processing, 2000. ICASSP'00.
  Proceedings. 2000 IEEE International Conference on}, volume~1, pages
  604--607. IEEE, 2000.

\bibitem[de~Castro et~al.(2016)de~Castro, Gassiat, and Lacour]{dCGL15}
Yohann de~Castro, {\'E}lisabeth Gassiat, and Claire Lacour.
\newblock Minimax adaptive estimation of nonparametric hidden {M}arkov models.
\newblock \emph{Journal of Machine Learning Research}, 17\penalty0
  (111):\penalty0 1--43, 2016.

\bibitem[De~Castro et~al.(2017)De~Castro, Gassiat, and Le~Corff]{dCGLLC15}
Yohann De~Castro, {\'E}lisabeth Gassiat, and Sylvain Le~Corff.
\newblock Consistent estimation of the filtering and marginal smoothing
  distributions in nonparametric hidden {M}arkov models.
\newblock \emph{IEEE Transactions on Information Theory}, 2017.

\bibitem[Dedecker et~al.(2007)Dedecker, Doukhan, Lang, Rafael, Louhichi, and
  Prieur]{dedecker2007weakdependence}
J{\'e}r{\^o}me Dedecker, Paul Doukhan, Gabriel Lang, Le{\'o}n R~Jos{\'e}
  Rafael, Sana Louhichi, and Cl{\'e}mentine Prieur.
\newblock \emph{Weak dependence: With examples and applications}.
\newblock Springer, 2007.

\bibitem[Douc and Matias(2001)]{douc2001asymptotics}
Randal Douc and Catherine Matias.
\newblock Asymptotics of the maximum likelihood estimator for general hidden
  {M}arkov models.
\newblock \emph{Bernoulli}, 7\penalty0 (3):\penalty0 381--420, 2001.

\bibitem[Douc and Moulines(2012)]{douc2012misspecified}
Randal Douc and {\'E}ric Moulines.
\newblock Asymptotic properties of the maximum likelihood estimation in
  misspecified hidden {M}arkov models.
\newblock \emph{The Annals of Statistics}, 40\penalty0 (5):\penalty0
  2697--2732, 2012.

\bibitem[Douc et~al.(2004)Douc, Moulines, and Ryd{\'e}n]{douc2004asymptotic}
Randal Douc, {\'E}ric Moulines, and Tobias Ryd{\'e}n.
\newblock Asymptotic properties of the maximum likelihood estimator in
  autoregressive models with {M}arkov regime.
\newblock \emph{The Annals of statistics}, 32\penalty0 (5):\penalty0
  2254--2304, 2004.

\bibitem[Douc et~al.(2009)Douc, Fort, Moulines, and
  Priouret]{douc2009forgetting}
Randal Douc, Gersende Fort, {\'E}ric Moulines, and Pierre Priouret.
\newblock Forgetting the initial distribution for hidden {M}arkov models.
\newblock \emph{Stochastic processes and their applications}, 119\penalty0
  (4):\penalty0 1235--1256, 2009.

\bibitem[Douc et~al.(2011)Douc, Moulines, Olsson, and
  Van~Handel]{douc2011consistency}
Randal Douc, {\'E}ric Moulines, Jimmy Olsson, and Ramon Van~Handel.
\newblock Consistency of the maximum likelihood estimator for general hidden
  {M}arkov models.
\newblock \emph{the Annals of Statistics}, 39\penalty0 (1):\penalty0 474--513,
  2011.

\bibitem[Douc et~al.(2020)Douc, Olsson, and Roueff]{douc2016posterior}
Randal Douc, Jimmy Olsson, and Fran{\c{c}}ois Roueff.
\newblock Posterior consistency for partially observed {M}arkov models.
\newblock \emph{Stochastic Processes and their Applications}, 130\penalty0
  (2):\penalty0 733--759, 2020.

\bibitem[Gerencs{\'e}r et~al.(2007)Gerencs{\'e}r, Michaletzky, and
  Moln{\'a}r-S{\'a}ska]{gerencser2007forgetting}
L{\'a}szl{\'o} Gerencs{\'e}r, Gy{\"o}rgy Michaletzky, and G{\'a}bor
  Moln{\'a}r-S{\'a}ska.
\newblock An improved bound for the exponential stability of predictive filters
  of hidden {M}arkov models.
\newblock \emph{Communications in Information \& Systems}, 7\penalty0
  (2):\penalty0 133--152, 2007.

\bibitem[Hsu et~al.(2012)Hsu, Kakade, and Zhang]{HKZ12}
Daniel Hsu, Sham~M Kakade, and Tong Zhang.
\newblock A spectral algorithm for learning hidden {M}arkov models.
\newblock \emph{Journal of Computer and System Sciences}, 78\penalty0
  (5):\penalty0 1460--1480, 2012.

\bibitem[Kruijer et~al.(2010)Kruijer, Rousseau, and Van
  Der~Vaart]{kruijer2010adaptive}
Willem Kruijer, Judith Rousseau, and Aad Van Der~Vaart.
\newblock Adaptive {B}ayesian density estimation with location-scale mixtures.
\newblock \emph{Electronic Journal of Statistics}, 4:\penalty0 1225--1257,
  2010.

\bibitem[Kunsch et~al.(1995)Kunsch, Geman, Kehagias,
  et~al.]{kunsch1995hiddenMarkovFields}
Hans Kunsch, Stuart Geman, Athanasios Kehagias, et~al.
\newblock Hidden markov random fields.
\newblock \emph{The Annals of Applied probability}, 5\penalty0 (3):\penalty0
  577--602, 1995.

\bibitem[Lambert et~al.(2003)Lambert, Whiting, and Metcalfe]{LWM03}
Martin~F Lambert, Julian~P Whiting, and Andrew~V Metcalfe.
\newblock A non-parametric hidden {M}arkov model for climate state
  identification.
\newblock \emph{Hydrology and Earth System Sciences Discussions}, 7\penalty0
  (5):\penalty0 652--667, 2003.

\bibitem[Le~Gland and Mevel(2000)]{legland2000forgetting}
Fran{\c{c}}ois Le~Gland and Laurent Mevel.
\newblock Exponential forgetting and geometric ergodicity in hidden {M}arkov
  models.
\newblock \emph{Mathematics of Control, Signals and Systems}, 13\penalty0
  (1):\penalty0 63--93, 2000.

\bibitem[Lef{\`e}vre(2003)]{Lef03}
Fabrice Lef{\`e}vre.
\newblock Non-parametric probability estimation for {HMM}-based automatic
  speech recognition.
\newblock \emph{Computer Speech \& Language}, 17\penalty0 (2):\penalty0
  113--136, 2003.

\bibitem[Leh{\'e}ricy(2018)]{lehericy2017sbs}
Luc Leh{\'e}ricy.
\newblock State-by-state minimax adaptive estimation for nonparametric hidden
  {M}arkov models.
\newblock \emph{The Journal of Machine Learning Research}, 19\penalty0
  (1):\penalty0 1432--1477, 2018.

\bibitem[Leroux(1992)]{leroux92MLEHMM}
Brian~G Leroux.
\newblock Maximum-likelihood estimation for hidden {M}arkov models.
\newblock \emph{Stochastic processes and their applications}, 40\penalty0
  (1):\penalty0 127--143, 1992.

\bibitem[Massart(2007)]{Mas07}
Pascal Massart.
\newblock Concentration inequalities and model selection.
\newblock In \emph{Lecture Notes in Mathematics}, volume 1896. Springer,
  Berlin, 2007.

\bibitem[Maugis-Rabusseau and Michel(2013)]{maugis2013adaptive}
Cathy Maugis-Rabusseau and Bertrand Michel.
\newblock Adaptive density estimation for clustering with {G}aussian mixtures.
\newblock \emph{ESAIM: Probability and Statistics}, 17:\penalty0 698--724,
  2013.

\bibitem[Merlev{\`e}de et~al.(2009)Merlev{\`e}de, Peligrad, and
  Rio]{merlevede2009bernstein}
Florence Merlev{\`e}de, Magda Peligrad, and Emmanuel Rio.
\newblock Bernstein inequality and moderate deviations under strong mixing
  conditions.
\newblock In \emph{High dimensional probability V: the Luminy volume}, pages
  273--292. Institute of Mathematical Statistics, 2009.

\bibitem[Mevel and Finesso(2004)]{mevel2004misspecified}
Laurent Mevel and Lorenzo Finesso.
\newblock Asymptotical statistics of misspecified hidden {M}arkov models.
\newblock \emph{IEEE Transactions on Automatic Control}, 49\penalty0
  (7):\penalty0 1123--1132, 2004.

\bibitem[Pouzo et~al.(2016)Pouzo, Psaradakis, and Sola]{pouzo2016misspecified}
Demian Pouzo, Zacharias Psaradakis, and Martin Sola.
\newblock Maximum likelihood estimation in possibly misspecified dynamic models
  with time inhomogeneous {M}arkov regimes.
\newblock \emph{arXiv preprint arXiv:1612.04932}, 2016.

\bibitem[Shen et~al.(2013)Shen, Tokdar, and Ghosal]{STG13bayesiandensity}
Weining Shen, Surya~T Tokdar, and Subhashis Ghosal.
\newblock Adaptive {B}ayesian multivariate density estimation with {D}irichlet
  mixtures.
\newblock \emph{Biometrika}, 100\penalty0 (3):\penalty0 623--640, 2013.

\bibitem[Vernet(2015{\natexlab{a}})]{vernet2015posterior}
{\'E}lodie Vernet.
\newblock Posterior consistency for nonparametric hidden {M}arkov models with
  finite state space.
\newblock \emph{Electronic Journal of Statistics}, 9\penalty0 (1):\penalty0
  717--752, 2015{\natexlab{a}}.

\bibitem[Vernet(2015{\natexlab{b}})]{vernet2015posteriorrates}
{\'E}lodie Vernet.
\newblock Non parametric hidden {M}arkov models with finite state space:
  posterior concentration rates.
\newblock \emph{arXiv preprint arXiv:1511.08624}, 2015{\natexlab{b}}.

\bibitem[Volant et~al.(2014)Volant, B{\'e}rard, Martin-Magniette, and
  Robin]{VBMMR14}
Stevenn Volant, Caroline B{\'e}rard, Marie-Laure Martin-Magniette, and
  St{\'e}phane Robin.
\newblock Hidden {M}arkov models with mixtures as emission distributions.
\newblock \emph{Statistics and Computing}, 24\penalty0 (4):\penalty0 493--504,
  2014.

\bibitem[Yau et~al.(2011)Yau, Papaspiliopoulos, Roberts, and Holmes]{YPRH11}
C~Yau, Omiros Papaspiliopoulos, Gareth~O Roberts, and Christopher Holmes.
\newblock Bayesian non-parametric hidden {M}arkov models with applications in
  genomics.
\newblock \emph{Journal of the Royal Statistical Society: Series B (Statistical
  Methodology)}, 73\penalty0 (1):\penalty0 37--57, 2011.

\end{thebibliography}

\appendix

\section{Proofs for the minimax adaptive estimation}
\label{sec_proof_minimax}

\subsection{Proofs for the mixture framework}

\subsubsection{Proof of Lemma~\ref{lemma_mixture_checking_assumptions} (checking the assumptions)}
\label{sec_preuve_lemma_mixture_checking_assumptions}

\paragraph{Checking \ref{Atail}}

By definition of the emission densities, $b_\gamma(y) \geq - 2 \log n$ for all $\gamma \in \Sbf_n^{(\gamma)}$. Moreover, for all $y \in \Ycal$ and $\gamma \in S_{K,M,n}^{(\gamma)}$,
\begin{align*}
b_\gamma(y)
	&\leq \log \left( \frac{1}{K} \sum_{x \in [K]} \left( 1 \vee \frac{\max_{\mu, s} \frac{1}{s} \psi\left(\frac{y - \mu}{s}\right)}{\Gref(y)} \right) \right) \\
	&\leq 0 \vee \left( \max_{\mu, s} \log \frac{1}{s} \psi\left(\frac{y - \mu}{s}\right) - \log \Gref(y) \right) \\
	&\leq 0 \vee \left( \max_{\mu, s} \left\{ \log \frac{1}{s} - \left( \frac{y-\mu}{s} \right)^p \right\} + \log(1 + y^2) + \log \frac{\pi}{2 \Gamma(1 + 1/p)} \right) \\
	&\leq 0 \vee \left( \log n - n^{-p} \min_{\mu} (y-\mu)^p + \log(1 + y^2) + \log \pi \right),
\end{align*}
where we recall that the maximum is taken over $\mu \in [-n,n]$ and $s \in [\frac{1}{n}, n]$.

If $y \in [-n,n]$,
\begin{align*}
b_\gamma(y)
	&\leq 0 \vee \left( \log n + \log(1 + y^2) + \log \pi \right) \\
	&\leq \log n + 1 + 2\log n + \log \pi \quad \text{since } n \geq 1\\
	&\leq 3 \log n + \log (\pi e) \leq 5 \log n
\end{align*}
as soon as $n \geq 3$. Otherwise, one can take $y \geq n$ and then
\begin{align*}
b_\gamma(y)
	&\leq 0 \vee ( \log n - n^{-p} (y-n)^p + \log(1 + y^2) + \log \pi ) \\
	&\leq 0 \vee ( \log n - n^{-p} (y-n)^p + \log(1 + 2(y-n)^2 + 2n^2) + \log \pi ) \\
	&\leq 0 \vee ( \log n - n^{-p} Y^p + \log(1 + 2Y^2) + 1 + \log 2n^2 + \log \pi )
\end{align*}
by writing $Y = y - n$ and using that $\log (a+b) \leq \log a + \log (1+b) \leq \log a + 1 + \log b$ when $a,b \geq 1$. Thus, writing $Y' = Y / n$,
\begin{align*}
b_\gamma(y)
	&\leq 3 \log n + \log (2e\pi) + 0 \vee ( - (Y')^p + \log(1 + 2n^2 (Y')^2) ) \\
	&\leq \begin{cases}
		3 \log n + \log (2e\pi) + \log(1 + 2n^2) \quad \text{if } Y' \leq 1 \\
		3 \log n + \log (2e\pi) + 0 \vee (- (Y')^p + 1 + \log(2n^2 (Y')^2) ) \quad \text{otherwise}
		\end{cases} \\
	&\leq 5 \log n + \log (4e^2\pi) + 0 \vee (- (Y')^p + 2\log(Y') ) \\
	&\leq 5 \log n + \log (4e^2\pi),
\end{align*}
so that $b_\gamma(y) \leq 10 \log n$ as soons as $n \geq 3$.

\paragraph{Checking \ref{Aentropy} and \ref{Agrowth}}

Let us first assume that there exists a constant $L_p$ such that the function $(\mu, s) \longmapsto \frac{s^{-1} \psi(s^{-1} (y-u))}{\Gref(y)}$ is $L_p$-Lipschitz for all $y$ (where the origin space is endowed with the supremum norm). Then a bracket covering of size $\epsilon$ of $([n, n] \times [\frac{1}{n}, n])^M$ provides a bracket covering of $\{ \gamma_x \}_{\gamma \in \Sbf_n^{(\gamma)}, x \in [K]}$ of size $L_p \epsilon$. Since there exists a bracket covering of size $\epsilon$ of $[n, n] \times [\frac{1}{n}, n]$ for the supremum norm with less than $(\frac{4n}{\epsilon} \vee 1)^2$ brackets, one gets \ref{Aentropy} by taking $\Caux(M, K, D, n) = 4 L_p n$ and $m_M = 2M$.

Let us now check that this constant $L_p$ exists.
\begin{align*}
\left| \frac{\partial}{\partial \mu} \frac{\frac{1}{s} \psi\left( \frac{y - \mu}{s} \right)}{\Gref(y)} \right|
	&= \frac{1}{2 \pi \Gamma(1 + \frac{1}{p}) (1 + y^2)} \left| \frac{\partial}{\partial \mu} \frac{1}{s} \exp\left( - \left(\frac{y - \mu}{s}\right)^p \right) \right| \\
	&= \frac{1}{2 \pi \Gamma(1 + \frac{1}{p}) (1 + y^2) s^2} \left|\frac{y - \mu}{s}\right|^{p-1} \exp\left( - \left(\frac{y - \mu}{s}\right)^p \right) \\
	&\leq \frac{1}{s^2} Y^{p-1} \exp(- Y^p) \\
	&\leq n^2 Z^{1 - 1/p} e^{- Z} \leq n^2
\end{align*}
by writing $Y = |y-\mu|/s$ and $Z = Y^p$. Likewise,
\begin{align*}
\left| \frac{\partial}{\partial s} \frac{\frac{1}{s} \psi\left( \frac{y - \mu}{s} \right)}{\Gref(y)} \right|
	&= \frac{1}{2 \pi \Gamma(1 + \frac{1}{p}) (1 + y^2)} \left| - \frac{1}{s^2} + p \frac{1}{s} \frac{(y - \mu)^p}{s^{p+1}} \right| \exp\left( - \left(\frac{y - \mu}{s}\right)^p \right) \\
	&\leq \frac{1}{s^2} | pZ - 1 | e^{-Z} \\
	&\leq n^2 \frac{p}{2}
\end{align*}
as soon as $p \geq 2$. Thus, one can take $L_p = p n^2$ and $\Caux(M, K, D, n) = 4 p n^3$. With this $\Caux$, checking \ref{Agrowth} is straightforward for all $\zeta > 0$: with $n \geq 4p$, it is ensured by $\log n^4 \leq n^\zeta$, which is always true for $\zeta = 2$ for instance.

\subsubsection{Proof of Lemma~\ref{lemma_mixture_approximation_rates} (approximation rates)}
\label{sec_proof_lemma_mixture_approximation_rates}

Let $F(y) = e^{-c|y|^\tau}$.
Lemma~4 of \cite{kruijer2010adaptive} ensures that there exist $c' > 0$ and $H \geq 6 \beta + 4 p$ such that for all $x \in [K^*]$ and $s > 0$, there exists a mixture $g_{s,x}$ with $O(s^{-1} |\log s|^{p / \tau})$ components, each with density $\frac{1}{s} \psi(\frac{\cdot - \mu}{s})$ with respect to the Lebesgue measure for some $\mu \in \{ y \, | \, F(y) \geq c' s^{H} \}$, such that $g_{s,x}$ approximates the emission density~$\gamma^*_x$:
\begin{equation*}
\max_x KL( \gamma^*_x \| g_{s,x} ) = O(s^{-2\beta}).
\end{equation*}

For $M$ large enough, the condition $M \geq O(s^{-1} | \log s |^{p/\tau})$ (number of components) is ensured by $s^{-1} \leq c M (\log M)^{-p/\tau}$ for some small enough constant $c > 0$. Take this $s$ in the following and $g_{M,x} = g_{s,x}$.

Let us now check that $(n^{-2} + (1 - n^{-2}) g_{M,x})_{x \in [K^*]} \in S^{(\gamma)}_{K^*,M,n}$.
When $M \leq n$ is large enough that $s \leq 1$, this $s$ is indeed in $[\frac{1}{n}, n]$.
When $|\mu| \geq s^{-1}$, $F(\mu) \leq \exp(-c s^{-\tau}) = o(c' s^H)$. Thus, for $s$ small enough (i.e. for $M$ large enough), all translation parameters $\mu$ belong to $[-s^{-1}, s^{-1}]$, which is indeed in $[-n,n]$ when $M \leq n$.

\subsubsection{Proof of Corollary~\ref{cor_mixture_minimax_estimation} (minimax adaptive estimation rate)}
\label{sec_proof_cor_mixture_minimax_estimation}

Denote by $h$ the Hellinger distance, defined by $h(p,q)^2 = \Ebb_P [(\sqrt{q/p} - 1)^2]$ for all probability densities $p$ and $q$ associated to probability measures $P$ and $Q$. Let
\begin{equation*}
\Hbf^2(K, \Qbf, \gamma) = \Ebb^*_{Y_{-\infty}^0} \left[ h^2(p^*_{Y_1 | Y_{-\infty}^0}, p_{Y_1 | Y_{-\infty}^0, (K, \Qbf, \gamma)}) \right]
\end{equation*}
be the Hellinger distance between the distributions of $Y_1$ conditionally to $Y_{-\infty}^0$ under the true distribution and under the parameters $(K, \Qbf, \gamma)$ (see Lemma~\ref{lemma_ecart_Likx} for the definition of these conditional distributions).

The following lemma shows that the Kullback-Leibler divergence and the Hellinger distance are equivalent up to a logarithmic factor and a small additive term.

\begin{lemma}
Assume that \ref{Astar_tail}, \ref{Astar_forgetting}, \ref{Atail} and \ref{Aergodic} hold.
Then there exists a constant $n_1$ depending on $\CB$, $\Csigma$, $\delta$ and $M_\delta$ such that for all $n \geq n_1$, 
for all $(K,\Qbf,\gamma) \in \Sbf_n$,
\begin{equation*}
\Hbf^2(K,\Qbf,\gamma)
	\leq \Kbf(K,\Qbf,\gamma)
	\leq 7 \CB (\log n)^2 \left( \Hbf^2(K,\Qbf,\gamma) + \frac{2}{n} \right).
\end{equation*}
\end{lemma}

\begin{proof}
The lower bound comes from the fact that the square of the Hellinger distance is smaller than the Kullback-Leibler divergence. For the upper bound, we use Lemma~4 of \cite{STG13bayesiandensity}: for all $v \geq 4$ and for all probability measures $P$ and $Q$ with densities $p$ and $q$,
\begin{equation*}
KL(p \| q) \leq h^2(p,q)\left( 1 + 2 v \right)
	+ 2 \Ebb_P \left[ \left(\log \frac{p}{q}\right) \one\left\{ \log \frac{p}{q} \geq v \right\} \right].
\end{equation*}

Take $p = p^*_{Y_1 | Y_{-\infty}^0}$ and $q = p_{Y_1 | Y_{-\infty}^0, (K,\Qbf,\gamma)}$. Then by equation~\eqref{eq_Li_encadre_par_d}, $\log \frac{p}{q} \leq |b_\gamma| + |L_{1,\infty}^*| + \log (\Csigma \log n)$ where $L_{1,\infty}^*$ is as in Lemma~\ref{lemma_ecart_Likx} and $\one\left\{ \log \frac{p}{q} \geq v \right\} \leq \one\{ |b_\gamma| \geq \frac{1}{2} (v - \log (\Csigma \log n) ) \} \vee \one\{ |L_{1,\infty}^*| \geq \frac{1}{2} (v - \log (\Csigma \log n) ) \}$. There exists $n_1$ depending only on $\CB$ and $\Csigma$ such that for all $n \geq n_1$, $\log (\Csigma \log n) \leq \CB (\log n)^2$. Assume $n \geq n_1$ and take $v = 3 \CB (\log n)^2$, then
 $\frac{1}{2}(v - \log (\Csigma \log n)) \geq (\CB \log n)^2$ and $1+2v \leq 7\CB (\log n)^2$, so that
\begin{align*}
\Kbf(K,\Qbf,\gamma)
	& \leq 7 \CB (\log n)^2 \Hbf^2(K,\Qbf,\gamma) \\
		+& 2 \CB (\log n)^2 \left\{ \Pbb^*( |b_\gamma| \geq \CB (\log n)^2 ) + \Pbb^*( |L_{1,\infty}^*| \geq \CB (\log n)^2 ) \right\} \\
		+& 2 \Ebb^* [ ( |L_{1,\infty}^*| + |b_\gamma| ) \\
		&\qquad \qquad \times (\one\{ |L_{1,\infty}^*| \geq \CB (\log n)^2 \} \vee \one\{ |b_\gamma| \geq \CB (\log n)^2 \}) ].
\end{align*}

By Lemma~\ref{lem_tail_moments}, which also holds for $L_{1,\infty}^*$ using the uniform convergence of Lemma \ref{lemma_ecart_Likx}, $\Pbb^*( |L_{1,\infty}^*| \geq \CB (\log n)^2 ) \leq \exp(-\log n) \leq n^{-1}$ for $n \geq n_0$ where $n_0$ is defined in Lemma~\ref{lem_tail_moments} (and depends on $\delta$ and $M_\delta$).
Likewise, by \ref{Atail}, ${\Pbb^*( |b_\gamma| \geq \CB (\log n)^2 ) \leq n^{-1}}$.

The last expectation of the above equation can be written as
\begin{align*}
2 \Ebb^* [ ( a + b ) \one\{ a \vee b \geq \CB (\log n)^2 \} ]
\end{align*}
where $a = |L_{1,\infty}^*|$ and $b = |b_\gamma|$. Then,
\begin{align*}
2 \Ebb^* [ a \one\{ a \vee b \geq \CB  &  (\log n)^2 \}) ] \\
	=& 2 \Ebb^* [ a \one\{ a \geq \CB (\log n)^2 \}) ]
		+ 2 \Ebb^* [ a \one\{ b \geq \CB (\log n)^2 > a \}) ] \\
	\leq{} & 4 \CB (\log n)^2 e^{- \log n}
		+ 2 \CB (\log n)^2 \Pbb^* [ b \geq \CB (\log n)^2 ] \\
	\leq{} & 6 \CB \frac{(\log n)^2}{n}
\end{align*}
by Lemma~\ref{lem_tail_moments} for the first term and \ref{Atail} for the second one. Likewise,
\begin{align*}
2 \Ebb^* [ b \one\{ a \vee b \geq \CB (\log n)^2 \}) ]
	\leq 6 \CB \frac{(\log n)^2}{n},
\end{align*}
so that finally
\begin{equation*}
\Kbf(K,\Qbf,\gamma)
	\leq 7 \CB (\log n)^2 \Hbf^2(K,\Qbf,\gamma) + 14 \CB \frac{(\log n)^2}{n},
\end{equation*}
which concludes the proof.
\end{proof}

Let $M \in \Nbb^*$. Let $g_{M,x}$ be the approximating densities given by Lemma~\ref{lemma_mixture_approximation_rates} and write $\gamma_{M,x} = n^{-2} + (1 - n^{-2}) g_{M,x}$ for all $x \in [K^*]$. The following lemma controls the error $\Hbf(K^*, \Qbf^*, (\gamma_{M,x})_x )$ coming from the approximation of the densities.
\begin{lemma}
Let $\sigma^* > 0$ be such that $\sigma^* \leq K^* \Qbf^*(x,x') \leq (\sigma^*)^{-1}$ for all $x,x' \in [K^*]$. Then
\begin{equation*}
\Hbf^2(K^*, \Qbf^*, (\gamma_{M,x})_x )
	\leq \left(2 + \frac{32 (K^*)^3}{(\sigma^*)^{11}} \right)
		\sum_{x \in [K^*]} h^2(\gamma^*_x, \gamma_{M,x})
\end{equation*}
\end{lemma}

\begin{proof}
Let $p^*_x = p^*(X_1 = x | Y_{-\infty}^0)$ and $p_x = p_{(K^*, \Qbf^*, (\gamma_{M,x})_x)}(X_1 = x | Y_{-\infty}^0)$. The Cauchy-Schwarz inequality implies that $(\sqrt{\sum_x a_x} - \sqrt{\sum_x b_x})^2 \leq \sum_x (\sqrt{a_x} - \sqrt{b_x})^2$, so that
\begin{align*}
h^2 \left(\sum_x p^*_x \gamma^*_x, \sum_x p_x \gamma_{M,x} \right)
	&= \int \left( \sqrt{\sum_x p^*_x \gamma^*_x} - \sqrt{\sum_x p_x \gamma_{M,x}} \right)^2 d\lambda \\
	\leq{} & \int \sum_x (\sqrt{p^*_x \gamma^*_x} - \sqrt{p_x \gamma_{M,x}})^2 d\lambda \\
	\leq{} & 2 \int \sum_x \left( p_x (\sqrt{\gamma^*_x} - \sqrt{\gamma_{M,x}})^2 + (\sqrt{p_x} - \sqrt{p^*_x})^2 \gamma^*_x \right) d\lambda \\
	\leq{} & 2 \sum_x p_x h^2 (\gamma^*_x, \gamma_{M,x} )
		+ 2 \sum_x (\sqrt{p^*_x} - \sqrt{p_x} )^2 \\
	\leq{} & 2 \sum_x h^2 (\gamma^*_x, \gamma_{M,x} )
		+ 2 \sum_x (\sqrt{p^*_x} - \sqrt{p_x} )^2
\end{align*}

Thus, one needs to control the expectation of the second term. 
Since $p_x$ and $p^*_x$ belong to $[\frac{\sigma^*}{K^*}, \frac{1}{K^* \sigma^*}]$ by assumption on $\Qbf^*$,
\begin{equation*}
\sum_x (\sqrt{p_x} - \sqrt{p^*_x})^2
	\in \left[ \frac{K^* \sigma^*}{4}, \frac{K^*}{4 \sigma^*} \right] \sum_x (p_x - p^*_x)^2.
\end{equation*}

The following equation follows from a careful reading of the proof of Proposition 2.1 of \cite{dCGLLC15} by noticing that the roles of $\gamma^*$ and $\gamma_M$ are symmetrical in their proof and that their reasoning works with $\rho_\star = 1 - \min \Qbf_\star / \max \Qbf_\star$.
\begin{equation*}
\sum_x | p_x - p^*_x |
	\leq \frac{4 K^*}{(\sigma^*)^3} \sum_{i = 0}^{+ \infty} (1 - (\sigma^*)^2)^i \frac{\max_x | \gamma^*_x(Y_{-i}) - \gamma_{M,x}(Y_{-i}) |}{\sum_x \gamma^*_x(Y_{-i}) \vee \sum_x \gamma_{M,x}(Y_{-i})}.
\end{equation*}

Therefore, using Cauchy-Schwarz's inequality:
\begin{align*}
\sum_x ( p_x - p^*_x )^2
	&\leq \left(\sum_x | p_x - p^*_x |\right)^2 \\
	&\leq \frac{16 (K^*)^2}{(\sigma^*)^8} \sum_{i = 0}^{+ \infty} (1 - (\sigma^*)^2)^i \left(\frac{\max_x | \gamma^*_x(Y_{-i}) - \gamma_{M,x}(Y_{-i}) |}{\sum_x \gamma^*_x(Y_{-i}) \vee \sum_x \gamma_{M,x}(Y_{-i}) }\right)^2.
\end{align*}

Since $\frac{|a-b|}{2 \sqrt{a \vee b}} \leq |\sqrt{a} - \sqrt{b}|$,
\begin{align*}
\Ebb^* \left(\frac{\max_x | \gamma^*_x(Y) - \gamma_{M,x}(Y) |}{\sum_x \gamma^*_x(Y) \vee \sum_x \gamma_{M,x}(Y) }\right)^2
	&\leq \int \frac{\max_x ( \gamma^*_x(y) - \gamma_{M,x}(y) )^2}{\sum_x \gamma^*_x(y) \vee \sum_x \gamma_{M,x}(y) } d\lambda(y) \\
	&\leq \sum_x \int \frac{( \gamma^*_x(y) - \gamma_{M,x}(y) )^2}{\gamma^*_x(y) \vee \gamma_{M,x}(y) } d\lambda(y) \\
	&\leq 4 \sum_x \int \left( \sqrt{\gamma^*_x(y)} - \sqrt{\gamma_{M,x}(y)} \right)^2 d\lambda(y) \\
	&= 4 \sum_x h^2(\gamma^*_x, \gamma_{M,x}),
\end{align*}
so that
\begin{align*}
\Ebb^* \left[\sum_x (\sqrt{p^*_x} - \sqrt{p_x} )^2\right]
	&\leq \frac{K^*}{4 \sigma^*} \Ebb^* \left[\sum_x ( p_x - p^*_x )^2 \right] \\
	&\leq \frac{16 (K^*)^3}{(\sigma^*)^{11}} \sum_x h^2(\gamma^*_x, \gamma_{M,x}),
\end{align*}
which concludes the proof of the lemma.
\end{proof}

Finally, since $|\sqrt{a+b} - \sqrt{c}| \leq |\sqrt{a} - \sqrt{c}| + \sqrt{|b|}$ for all $b \in \Rbb$, $a \geq (-b) \vee 0$ and $c \geq 0$, for all $x$,
\begin{align*}
h^2(\gamma^*_x, \gamma_{M,x} ) &\leq 2 h^2(\gamma^*_x, g_{M,x} ) + \frac{4}{n^2} \\
	&\leq 2 KL(\gamma^*_x \| g_{M,x}) + \frac{4}{n^2}.
\end{align*}

Therefore,
\begin{multline*}
\Kbf(K^*, \Qbf^*, (\gamma_{M,x})_x) \leq 14 \CB \frac{(\log n)^2}{n} \\
	+ 7\CB (\log n)^2 \left( 2 + \frac{32 (K^*)^3}{(\sigma^*)^{11}} \right) \sum_{x \in [K^*]} \left( \frac{4}{n^2} + 2 KL(\gamma^*_x, g_{M,x}) \right).
\end{multline*}

Thus, there exists a constant $C$ such that for all $n \geq 3$,
\begin{equation*}
\Kbf(K^*, \Qbf^*, (\gamma_{M,x})_x) \leq C (\log n)^2 \left(\frac{1}{n} + M^{-2\beta} (\log M)^{2\beta \frac{p}{\tau}} \right)
\end{equation*}
by definition of the densities $g_{M,x}$.

The choice of penalty verifies the lower bound of Theorem~\ref{th_oracle_simplifie}. Thus, the oracle inequality of Theorem~\ref{th_oracle_simplifie} with $\eta = 1$, $\alpha = 2$ and $t = 2 \log n$ entails that for $n$ large enough and for any sequence $(M_n)_n$ such that $K^* \leq M_n \leq n/2$ for all $n$:
\begin{align*}
\Kbf(\hat{K}, \hat{\Qbf}, \hat{\gamma})
	\leq{} & 2 \Kbf(K^*, \Qbf^*, (\gamma_{M_n,x})_x) + 4 \pen_n(K^*, M_n) + A \frac{(\log n)^{10}}{n} \\
	\leq{} & 2 C (\log n)^2 \left(\frac{1}{n} + M_n^{-2\beta} (\log n)^{2\beta \frac{p}{\tau}} \right) \\
		&+ 4 K^* \frac{(\log n)^{18}}{n} M_n + 2A \frac{(\log n)^{10}}{n}.
\end{align*}

Taking $M_n \sim n^{\frac{1}{2\beta + 1}} (\log n)^{\frac{2\beta p/\tau - 16}{2\beta + 1}}$ leads to the desired rate.

\section{Proof of the control of \texorpdfstring{$\bar{\nu}_k$}{nu\_k} (Theorem~\ref{th_penalite_L2})}
\label{sec_proof_control_barnuk}

Let us give an overview of the proof of the control of $\bar{\nu}_k$.

The first step of the proof is to obtain a Bernstein inequality on $\bar{\nu}_k(t)$ for a single function $t$. This is done using the mixing properties of the process $(Y_i)_i$ and by noticing that $\bar{\nu}_k(t)$ is the deviation of an empirical mean.

The second step is to transform the inequality on one function $t$ into an inequality on the supremum over all function $t$ belonging to a given class. This step involves the bracketing entropy of the aforementionned class. The control of this entropy is where the shape of the penalty appears.

At this stage, one is able to upper bound the supremum of $\bar{\nu}_k(t^{(D)}_{(K, \Qbf, \gamma)})$ over all parameters ${(K, \pi, \Qbf, \gamma) \in S_{K,M,n}}$. However, this upper bound is of order $n^{-1/2}$ (up to logarithmic factors), which is suboptimal. The third step of the proof gets rid of the $n^{-1/2}$ term by considering the processes
\begin{equation*}
W_{K,M,n} :=
\sup_{(K, \pi, \Qbf, \gamma) \in S_{K,M,n}} \frac{| \bar{\nu}_k(t_{(K, \Qbf, \gamma)}^{(D)}) |}{\Ebb^*[t_{(K, \Qbf, \gamma)}^{(D)}(Z_0)^2] + x_{K,M,n}^2}
\end{equation*}
for some constants $x_{K,M,n}$. The last step of the proof consists in taking appropriate $x_{K,M,n}$ in order to have with high probability and for all $K$ and $M$
\begin{equation*}
\begin{cases}
W_{K,M,n} \leq \epsilon \\
W_{K,M,n} x_{K,M,n}^2 \leq \pen_n(K,M) + R_n
\end{cases}
\end{equation*}
for a residual term $R_n$ depending on the probability, which leads to the desired inequality
\begin{equation*}
\forall (K, \pi, \Qbf, \gamma) \in S_{K,M,n}, \  | \bar{\nu}_k(t_{(K, \Qbf, \gamma)}^{(D)}) | - \pen_n(K,M)
	\leq \epsilon \Ebb^*[t_{(K, \Qbf, \gamma)}^{(D)}(Z_0)^2] + R_n.
\end{equation*}

The concentration results are stated in Section~\ref{sec_proofconcentration_concentration}. The control of the bracketing entropy is done in Section~\ref{sec_control_bracketing}. Finally, the choice of $x_{K,M,n}$ and the synthesis of the proof are done in Section~\ref{sec_choice_parameters}.

Without loss of generality, we assume $n \geq \exp(\Csigma)$ and $D \geq \log n$ so that $\|t^{(D)}_{(K, \Qbf, \gamma)}\|_\infty \leq 4D$ for all $(K, \pi, \Qbf, \gamma) \in S_{K,M,n}$ by equation~\eqref{eq_Li_borne_par_d} and $n$ larger than the constant $n_0$ from Lemma~\ref{lem_tail_moments}.

\paragraph{Changes of notations.}

In the rest of this section, we omit the dependency of $W_{K,M}$, $x_{K,M}$ and $S_{K,M}$ on $n$ in the notations. We also introduce the notation $\theta \in \Sbf_n$ instead of $(K,\pi,\Qbf,\gamma) \in \Sbf_n$ to make the notation shorter. Given $\theta  \in \Sbf_n$, we write $\pi_\theta$, $\Qbf_\theta$ and $\gamma_\theta$ its components. To avoid multiple subscripts, we write $\gamma_\theta(y | x)$ instead of $\gamma_{\theta,x}(y)$.

\subsection{Concentration inequality}
\label{sec_proofconcentration_concentration}

First, let us introduce some notations.
Let $D > 0$, $K \geq 1$, $M \in \Mcal$ and $k \geq 1$. For all $i \in \Zbb$, let $Z_i = Y_{i-k}^i$. Define for all $\sigma > 0$ the sets
\begin{equation*}
\Bbf_\sigma = \{ \theta \in S_{K,M} \; | \; \Ebb^*[t_\theta^{(D)}(Z_0)^2] \leq \sigma^2 \}.
\end{equation*}

Let $d_k$ be the semi-distance defined by $d_k^2(t_1, t_2) = \Ebb^*[(t_1 - t_2)^2(Z_0)]$. For any semi-distance $d$, write $N(A, d, \epsilon) = e^{H(A, d, \epsilon)}$ the minimal cardinality of a covering of $A$ by brackets of size $\epsilon$ for the semi-distance $d$, that is by sets $[t_1, t_2] = \{ {t : \Ycal^k \longmapsto \Rbb} \, , \, t_1(\cdot) \leq t(\cdot) \leq t_2(\cdot) \}$ such that $d(t_1, t_2) \leq \epsilon$. $H(A, d, \cdot)$ is called the \emph{bracketing entropy} of $A$ for the semi-distance $d$.

The first step of the proof is to obtain a Bernstein inequality for the deviations of a single $t^{(D)}(Z_i)$.
\begin{theorem}
Assume \ref{Astar_mixing} holds. Then there exists a constant $C_\mix$ depending on $c_*$ and $n_*$ such that the following holds.

Let $t$ be a real valued, measurable bounded function on $\Ycal^{k+1}$ and let $V = \Ebb^* [t^2(Z_0)]$. Then for all $\lambda \in (0, \frac{1}{C_\mix (n_*+k+1) \| t \|_\infty (\log n)^2})$ and for all $n \in \Nbb$:
\begin{align*}
\phi(\lambda) :={} & \log \Ebb^* \exp \left[ \lambda \sum_{i=1}^{n}\left(t(Z_i)-\Ebb^* t(Z_i)\right)  \right] \\
	\leq{} & \frac{C_\mix^2 (n_*+k+1)^2 (n V + \| t \|_{\infty}^2) \lambda^2}{1 - C_\mix (n_*+k+1) \| t \|_{\infty} (\log n)^2 \lambda}
\end{align*}
\end{theorem}

\begin{proof}
The following result is a Bernstein inequality for exponentially $\alpha$-mixing processes.

\begin{lemma}[\cite{merlevede2009bernstein}, Theorem 2]
\label{lemma_BernsteinAlphaMixing}
Let $(A_i)_{i \geq 1}$ be a stationary sequence of centered real-valued random variables such that $\| A_1 \|_\infty \leq M$ and whose $\alpha$-mixing coefficients satisfy, for a certain $c > 0$,
\begin{equation*}
\forall n \in \Nbb, \qquad \alpha_\mix(n) \leq e^{-2cn}.
\end{equation*}

Then there exist positive constants $C_1$ and $C_2$ depending on $c$ such that for all $n \geq 2$ and all $\lambda \in (0, \frac{1}{C_1 M (\log n)^2} )$,
\begin{equation*}
\log \Ebb \exp \left[ \lambda \sum_{i=1}^{n} A_i \right]
	\leq \frac{C_2 \lambda^{2} (n v + M^2)}{1 - C_1 \lambda M (\log n)^2},
\end{equation*}
where $v$ is defined by
\begin{equation*}
v = \Var(A_1) + 2 \sum_{i > 1} | \Cov \, (A_1, A_i) |.
\end{equation*}
\end{lemma}

Assumption \ref{Astar_mixing} implies that the $\alpha$-mixing coefficients of $(Y_i)_i$ satisfy $\alpha_\mix(n) \leq e^{-c_* n}$ for all $n \geq n_*$ since $4\alpha_\mix(n) \leq \rho_\mix(n)$ (see for instance \cite{bradley2005strongmixingsurvey}). However, this is not enough to apply the previous result: one needs the inequality to hold for all $n$ (and not for $n$ larger than some constant) and for the process $(Z_i)_i$. To do so, we partition the process $(Z_i)_i$ into several processes for which the above result applies, and then gather the inequalities.

Consider the processes $(Z_{i (n_*+k+1) + j})_i$ with $\alpha$-mixing coefficients $\alpha_{Z,j}(n)$. By construction, they satisfy $\alpha_{Z,j}(n) \leq e^{-c_* n_* n}$ for all $n \geq 1$ and $j \in \{1, \dots, n_*+k+1\}$. Apply Lemma~\ref{lemma_BernsteinAlphaMixing}, one gets that there exist two positive constants $C_1$ and $C_2$ depending on $c_*$ and $n_*$ such that for all functions $t$, all $\lambda \in (0, \frac{1}{C_1 M (\log n)^2})$ and all $n \in \Nbb$:
\begin{align*}
\phi_j(\lambda) &:= \log \Ebb^* \exp \left[ \lambda \sum_{i=1}^{n} ( t(Z_{i (n_*+k+1) + j}) - \Ebb t(Z_{i (n_*+k+1) + j}) ) \right] \\
	&\leq \frac{C_2 \lambda^{2} (n v + \| t \|_{\infty}^2)}{1 - C_1 \lambda \| t \|_{\infty} (\log n)^2}
\end{align*}
where, denoting $V = \Ebb^* t^2(Z_0)$:
\begin{align*}
v &= \Var(t(Z_j)) + 2 \sum_{i > 1} | \Cov \, (t(Z_j), t(Z_{i (n_*+k+1) + j}) | \\
	&\leq V + 2 V \sum_{i > 1} |\Corr \, (t(Z_j), t(Z_{i (n_*+k+1) + j}) | \\
	&\leq V \left(1 + 8 \sum_{i > 1} e^{-c_* n_* i}\right)
	 \leq \frac{8 V}{1 - e^{-c_* n_*}}
\end{align*}
using \ref{Astar_mixing}.
Finally, using that $\Ebb \prod_{i=1}^k A_i \leq \prod_{i=1}^k (\Ebb A_i^k)^{1/k}$ for any positive integer $k$ and any positive random variable $(A_i)_{1 \leq i \leq k}$,
\begin{equation*}
\phi(\lambda) \leq \frac{1}{n_*+k+1} \sum_{j=1}^{n_*+k+1} \phi_j((n_*+k+1)\lambda),
\end{equation*}
so that
\begin{equation*}
\phi(\lambda)
	\leq \frac{\frac{8 C_2}{1 - e^{-c_* n_*}} (n_*+k+1)^2 \lambda^{2} (n V + \| t \|_{\infty}^2)}{1 - C_1 (n_*+k+1) \lambda \| t \|_{\infty} (\log n)^2},
\end{equation*}
which concludes the proof.
\end{proof}

The following result follows \textit{mutatis mutandis} from the proof of Theorem 6.8 of \cite{Mas07} using the previous theorem.
\begin{lemma}
\label{lemma_inegalite_concentration}
Assume \ref{Astar_mixing} holds. Then there exists a constant $C^* \geq 1$ depending on $n_*$ and $c_*$ such that the following holds.

Let $\Tcal$ be a class of real valued and measurable functions on $\Ycal^{k+1}$ such that $\Tcal$ is separable for the supremum norm.
Also assume that there exist positive numbers $\sigma$ and $b$ such that for all $t \in \Tcal$, $\| t \|_\infty \leq b$ and $\Ebb^* t^2(Z_0) \leq \sigma^2$ and assume that $N(\Tcal,d_k,\delta)$ is finite for all $\delta > 0$.

Then for all measurable sets $A$ such that $\Pbb^*(A) > 0$:
\begin{multline*}
\Ebb^{*} \left(\underset{t \in \Tcal}{\sup} |\bar{\nu}_k(t)| \Big| A \right)
	\leq C^* (n_*+k+1) \Bigg[
		\frac{E}{n}
		+ \sigma \sqrt{\frac{1}{n} \log \left(\frac{1}{\Pbb^*(A)} \right)} \\
		+ \frac{b (\log n)^2}{n} \log \left(\frac{1}{\Pbb^*(A)} \right)
	\Bigg]
\end{multline*}
where
\begin{equation*}
E = \sqrt{n} \int_0^\sigma \sqrt{H(\Tcal, d_k, u) \wedge n} du + b (\log n)^2 H(\Tcal, d_k, \sigma).
\end{equation*}
\end{lemma}

By taking $\Tcal = \{ t^{(D)}_\theta \; | \theta \in \Bbf_\sigma\}$ and $b = 4D$, one gets the following lemma from Lemma 4.23 and Lemma 2.4 of \cite{Mas07}:
\begin{lemma}
\label{lemma_majorationW}
Assume that there exist a function $\varphi$ and constants $C$ and $\sigma_{K,M}$ such that $x \mapsto \frac{\varphi(x)}{x}$ is nonincreasing and
\begin{equation}
\label{majoration_E}
\forall \sigma \geq \sigma_{K,M} \qquad E \leq C \varphi(\sigma) \sqrt{n}.
\end{equation}
Then for all $x_{K,M} \geq \sigma_{K,M}$ and $z > 0$, with probability greater than $1 - e^{-z}$:
\begin{multline}
\label{majoration_Z_choix_x_a_faire}
W_{K,M} :=
\sup_{\theta \in S_{K,M}} \left| \frac{| \bar{\nu}_k(t_\theta^{(D)}) |}{\Ebb^*[t_\theta^{(D)}(Z_0)^2] + x_{K,M}^2} \right|
	\leq 4 C^* (n_*+k+1) \Bigg[
		C \frac{\varphi(x_{K,M})}{x_{K,M}^2 \sqrt{n}} \\
		+ \sqrt{\frac{z}{x_{K,M}^2 n}}
		+ 4D \frac{z (\log n)^2}{x_{K,M}^2 n}
	\Bigg].
\end{multline}
\end{lemma}
The two remaining steps are the control of the bracketing entropy which will lead to equation~\eqref{majoration_E} (see Section~\ref{sec_control_bracketing}) and the choice of the parameters $x_{K,M}$ and $z$ (see Section~\ref{sec_choice_parameters}).

\subsection{Control of the bracketing entropy}
\label{sec_control_bracketing}

In this section, we show that for all $k \geq 2$ and $\epsilon > 0$,
\begin{multline*}
H(\epsilon) \leq 2 ( m_M K + K^2 - 1) \log \max \Bigg( 
		\frac{95 D e^{2D}  \left( \sqrt{2} \Csigma \log n \right)^{k+3/2} \sqrt{k K \Caux'}}{\epsilon}, \\
		14 \left( \sqrt{2} \Csigma \log n \right)^{k+1/2} \sqrt{k K \Caux'} \Bigg)
\end{multline*}
where $\Caux' = (\Caux e^D) \vee (K-1)$.

\subsubsection{Reduction of the set}

For all $\theta \in S_{K,M}$, let $\gbf_\theta = (g_{\theta,x})_{x \in [K]}$ where
\begin{equation*}
g_{\theta,x}(y_0^k) = \!
\begin{cases}
\displaystyle p_\theta(X_k = x, Y_k = y_k | Y_0^{k-1} = y_0^{k-1}) \text{ if } |L_{k,k}^*| \vee \sup_{\theta' \in \Sbf_n} | b_{\theta'}(y_k) | \leq D, \\
\displaystyle 0 \quad \text{otherwise.}
\end{cases}
\end{equation*}

In order to control the bracketing entropy of $\{t^{(D)}_\theta \; | \; \theta \in \Bbf_\sigma\}$, we control the bracketing entropy of the set $\Gcal := \{ \gbf_\theta \; | \; \theta \in S_{K,M} \}$ for the distance
\begin{multline*}
d_\Gcal(\gbf_{\theta_1}, \gbf_{\theta_2}) = \Ebb_{Y_0^{k-1}}^* \Bigg[  \sum_{x \in [K]} \int | g_{\theta_1,x}(Y_0^{k-1}, y_k) - g_{\theta_2,x}(Y_0^{k-1}, y_k) | \\
\times \one_{|L_{k,k}^*| \vee \sup_{\theta' \in \Sbf_n} | b_{\theta'}(y_k) | \leq D} d \lambda(y_k) \Bigg].
\end{multline*}

\begin{remark}
\label{remark_yk}
In the rest of Section~\ref{sec_control_bracketing}, we always assume that 
\begin{equation}
\label{eq_rk_condition_yk}
|L_{k,k}^*| \vee \sup_{\theta' \in \Sbf_n} | b_{\theta'}(y_k) | \leq D
\end{equation}
since if this is not the case, then ${t_\theta^{(D)}(y_k) = t_{\theta'}^{(D)}(y_k)} = 0$. This means that only the $y_k$ satisfying equation~\eqref{eq_rk_condition_yk} are relevant for the construction of the brackets.
\end{remark}

For all $\theta \in S_{K,M}$,
\begin{align*}
\sum_{x \in [K]} g_{\theta,x}
	&= \sum_{x,x' \in [K]} p_\theta(Y_k = y_k | X_k = x) \Qbf_\theta(x',x) p_\theta(X_{k-1} = x' | Y_0^{k-1} = y_0^{k-1}) \\
	&\in \Bigg[(\Csigma \log n)^{-1} K^{-1} \sum_{x \in [K]} p_\theta(Y_k = y_k | X_k = x), \\
		&\hspace{3cm} \Csigma (\log n) K^{-1} \sum_{x \in [K]} p_\theta(Y_k = y_k | X_k = x) \Bigg] \\
	&= \left[(\Csigma \log n)^{-1} e^{b_\theta(y_k)}, \Csigma (\log n) e^{b_\theta(y_k)} \right],
\end{align*}
so that for all $\theta \in S_{K,M}$,
\begin{equation}
\label{eq_encadrement_gtheta}
(\Csigma (\log n) e^{D})^{-1} \leq \sum_{x \in [K]} g_{\theta,x} \leq \Csigma (\log n) e^{D}.
\end{equation}

Let $[a,b]$ be a bracket of size $\epsilon$ for $\Gcal$ with the distance $d_\Gcal$ such that
\begin{equation}
\label{eq_bornitude_crochet}
(2 \Csigma (\log n) e^{D})^{-1} \leq \sum_x a_x \leq \sum_x b_x \leq 2 \Csigma (\log n) e^{D}.
\end{equation}

Then
\begin{align*}
\left( \log \sum_x a_x - \log \sum_x b_x \right)^2
	&\leq 2 \log\left( 2\Csigma (\log n) e^D \right) \left| \log \sum_x a_x - \log \sum_x b_x \right| \\
	&\leq 8D \times 2 \Csigma (\log n) e^{D} \sum_x |a_x - b_x|
\end{align*}
when $n \geq e^2$ using that $|\log a - \log b| \leq |a-b|/(a \wedge b)$.
Therefore,
\begin{align*}
d_k & \left(\log \sum_x a_x, \log \sum_x b_x \right)^2 \\
	&= \Ebb^*_{Y_0^{k-1}} \left[ \int \left(\log \sum_x a_x - \log \sum_x b_x \right)^2(Y_0^{k-1}, y_k) p^*(Y_k = y_k | Y_0^{k-1}) \lambda(dy_k) \right] \\
	&\leq 16 D \Csigma (\log n) e^{D} \Ebb^*_{Y_0^{k-1}} \left[ \int \sum_x |a_x - b_x|(Y_0^{k-1}, y_k) \exp(L^*_{k,k}) \lambda(dy_k) \right] \\
	&\leq 16 D \Csigma (\log n) e^{2D} d_\Gcal(a,b),
\end{align*}
so that
\begin{equation}
\label{eq_lien_entropie_Bsigma_et_G}
N(\{t^{(D)}_\theta \; | \; \theta \in \Bbf_\sigma\}, d_k, \epsilon) \leq \bar{N} \left(\Gcal, d_\Gcal, \left( \frac{\epsilon}{16 D \Csigma (\log n) e^{2D}} \right)^2 \right)
\end{equation}
where $\bar{N}$ is the minimal cardinality of a bracket covering of $\Gcal$ such that all brackets $[a,b]$ satisfy equation~\eqref{eq_bornitude_crochet}.

\subsubsection{Decomposition into simple sets}

The aim of this section is to prove the following lemma.
\begin{lemma}
\label{lemma_lien_entropie_Gcal_et_ens_simples}
Assume $k \geq 2$ and let $\displaystyle \epsilon \in \left(0, \frac{70}{168} \right)$.  Then
\begin{multline*}
\bar{N} \left(\Gcal, d_\Gcal, \epsilon \right)
	\leq N\left(\{ \pi_\theta \}_{\theta \in S_{K,M}}, d_\infty, \frac{\epsilon}{70 k \left( \sqrt{2} \Csigma \log n \right)^{2k+1} K} \right) \\
		\times N\left(\{ \Qbf_\theta \}_{\theta \in S_{K,M}}, d_\infty, \frac{\epsilon}{70 k \left( \sqrt{2} \Csigma \log n \right)^{2k+1} K} \right) \\
		\times N\left(\{ \gamma_\theta \}_{\theta \in S_{K,M}}, d_\infty, \frac{\epsilon \, e^{-D}}{70 k \left( \sqrt{2} \Csigma \log n \right)^{2k+1} K} \right)
\end{multline*}
where $d_\infty$ is the distance of the supremum norm and where $\gamma_\theta$ denotes the function $(x,y) \longmapsto \gamma_\theta(y | x)$.
\end{lemma}

\noindent
Let:
\begin{itemize}
\item $[a,b]$ be a bracket of $\{ \pi_\theta \}_{\theta \in S_{K,M}}$ of size $\epsilon$ for the supremum norm;
\item $[p,q]$ be a bracket of $\{ \Qbf_\theta \}_{\theta \in S_{K,M}}$ of size $\epsilon$ pour the supremum norm;
\item $[u,v]$ be a bracket of $\{ \gamma_\theta \}_{\theta \in S_{K,M}}$ of size $\epsilon e^{-D}$ for the supremum norm.
\end{itemize}
Without loss of generality, we assume $(\Csigma \log n)^{-1} K^{-1} \leq a(x) \leq b(x) \leq \Csigma (\log n) K^{-1}$ and $(\Csigma \log n)^{-1} K^{-1} \leq p(x,x') \leq q(x,x') \leq \Csigma (\log n) K^{-1}$ for all $x, x' \in [K]$ since all elements of $\{ \pi_\theta \}_{\theta \in S_{K,M}}$ and $\{ \Qbf_\theta \}_{\theta \in S_{K,M}}$ satisfy these inequalities. We also assume that the brackets aren't empty: there exists $\theta \in S_{K,M}$ such that $\pi_\theta \in [a,b]$, $\Qbf_\theta \in [p,q]$ and $\gamma_\theta \in [u,v]$. Under this assumption, for all $y \in \Ycal$, 
\begin{equation}
\label{eq_encadrement_bracketuv}
Ke^{-D} (1 - \epsilon) \leq \sum_x u(y|x) \leq \sum_x v(y|x) \leq K(e^D + \epsilon e^{-D}).
\end{equation}

Using the approach of Appendix A of \cite{dCGLLC15}, one can write $g_{\theta,x}$ as the following product of matrices
\begin{equation*}
g_{\theta, x}(y_0^k) = \left( \mu_{0|k-1}^\theta F_{1|k-1}^\theta \dots F_{k-1|k-1}^\theta \Qbf_\theta \right)_x \gamma_\theta(y_k | x) 
\end{equation*}
where
\begin{align*}
&\beta_{i|k}(x_i) = \sum_{x_{i+1}^k \in [K]^{k-i}}
	\Qbf_\theta(x_i, x_{i+1}) \gamma_\theta(y_{i+1} | x_{i+1}) \dots
	\Qbf_\theta(x_{k-1}, x_k) \gamma_\theta(y_k | x_k),
\end{align*}
for $0 \leq i \leq k-1$ and $\beta_{k|k}(x) = 1$ for all $x \in [K]$,
\begin{align*}
&\mu_{0|k}^\theta(x) = \frac{\pi_\theta(x) \beta_{0|k}(x) \gamma_\theta(y_0 | x)}{\sum_{x' \in [K]} \pi_\theta(x') \beta_{0|k}(x') \gamma_\theta(y_0 | x')} \\
\text{and} \quad & F_{i|k}^\theta(x_{i-1}, x_i) =
	\frac{\beta_{i|k}(x_i) \Qbf_\theta(x_{i-1}, x_i) \gamma_\theta(y_i | x_i)}
		{\sum_{x \in [K]} \beta_{i|k}(x) \Qbf_\theta(x_{i-1}, x) \gamma_\theta(y_i | x)}.
\end{align*}

To clarify the role of these quantities, observe that
\begin{align*}
&\beta_{i|k}(x_i) = p_\theta(Y_{i+1}^k | X_i = x_i), \\
&\mu_{0|k}^\theta(x) = \Pbb_\theta(X_0 = x | Y_0^k), \\
&F_{i|k}^\theta(x_{i-1}, x_i) = \Pbb_\theta(X_i = x_i | Y_i^k, X_{i-1} = x_{i-1}),
\end{align*}
so that
\begin{equation*}
\left( \mu_{0|k}^\theta F_{1|k}^\theta \dots F_{k|k}^\theta \right)_x
	= \Pbb_\theta(X_k = x | Y_0^k).
\end{equation*}

Now, let
\begin{equation*}
\begin{cases}
\displaystyle \alpha_{i|k}(x_i) = \sum_{x_{i+1}^k \in [K]^{k-i}}
	p(x_i, x_{i+1}) u(y_{i+1} | x_{i+1}) \dots
	p(x_{k-1}, x_k) u(y_k | x_k) \\ \\
\displaystyle \delta_{i|k}(x_i) = \sum_{x_{i+1}^k \in [K]^{k-i}}
	q(x_i, x_{i+1}) v(y_{i+1} | x_{i+1}) \dots
	q(x_{k-1}, x_k) v(y_k | x_k)
\end{cases}
\end{equation*}
for $0 \leq i \leq k-1$ and $\alpha_{k|k}(x) = \delta_{k|k}(x) = 1$ for all $x \in [K]$,
\begin{equation*}
\begin{cases}
\displaystyle \nu(x) = \frac{a(x) \alpha_{0|k}(x) u(y_0|x)}{\sum_{x' \in [K]} b(x') \delta_{0|k}(x') v(y_0|x')} \\ \\
\displaystyle \omega(x) = \frac{b(x) \delta_{0|k}(x) v(y_0|x)}{\sum_{x' \in [K]} a(x') \alpha_{0|k}(x') u(y_0|x')}
\end{cases}
,
\end{equation*}
and
\begin{equation*}
\begin{cases}
\displaystyle f_{i|k}(x_{i-1}, x_i) =
	\frac{\alpha_{i|k}(x_i) p(x_{i-1}, x_i) u(y_i | x_i)}
		{\sum_{x \in [K]} \delta_{i|k}(x) q(x_{i-1}, x) v(y_i | x)} \\ \\
\displaystyle g_{i|k}(x_{i-1}, x_i) =
	\frac{\delta_{i|k}(x_i) q(x_{i-1}, x_i) v(y_i | x_i)}
		{\sum_{x \in [K]} \alpha_{i|k}(x) p(x_{i-1}, x) u(y_i | x)}
\end{cases}.
\end{equation*}
$[\nu, \omega]$ and $[f_{i|k},g_{i|k}]$ are brackets of $\{ \mu_{0|k}^\theta \}_{\theta \in S_{K,M}}$ and $\{ F_{i|k}^\theta \}_{\theta \in S_{K,M}}$ for all $i \in \{1, \dots, k \}$. Moreover, if one has a bracket covering of the sets $\{ \pi_\theta \}_{\theta \in S_{K,M}}$, $\{ \Qbf_\theta \}_{\theta \in S_{K,M}}$ and $\{ \gamma_\theta \}_{\theta \in S_{K,M}}$, then this construction gives a bracket covering of $\{ \mu_{0|k}^\theta \}_{\theta \in S_{K,M}}$ and $\{ F_{i|k}^\theta \}_{\theta \in S_{K,M}}$ for all $i \in \{1, \dots, k \}$.

The next step of the proof is to control the size of these new brackets.

\begin{lemma}
\label{lemma_taille_crochet_beta}
Assume $\epsilon \leq \frac{1}{2}$, then
\begin{equation*}
\sup_{0 \leq i \leq k} \frac{\sum_{x \in [K]} |\alpha_{i|k}(x) u(y_i | x) - \delta_{i|k}(x) v(y_i | x)|}{\sum_{x \in [K]} \alpha_{i|k}(x) u(y_i | x)}
	\leq 4 \left( \sqrt{2} \Csigma \log n \right)^{2k+1} K \epsilon.
\end{equation*}
\end{lemma}
\begin{proof}
Using minimalist notations,
\begin{align*}
\sum_{x \in [K]} &| \alpha_{i|k}(x) u(y_i | x) - \delta_{i|k}(x) v(y_i | x) | \\
	&\leq \sum_{j=i+1}^k \sum_{x_i^k \in [K]^{k-i+1}} u_i p_i^{i+1} u_{i+1} \dots u_{j-1} | p_{j-1}^j - q_{j-1}^j | v_j \dots q_{k-1}^k v_k \\
	& \quad + \sum_{j=i}^k \sum_{x_i^k \in [K]^{k-i+1}} u_i p_i^{i+1} u_{i+1} \dots p_{j-1}^j | u_j - v_j | q_j^{j+1} \dots q_{k-1}^k v_k.
\end{align*}
Then, note that for all $j \in \{i+1, \dots, k\}$,
\begin{align*}
\sum_{x_i^k \in [K]^{k-i+1}} & u_i p_i^{i+1} \dots p_{j-2}^{j-1} u_{j-1} | p_{j-1}^j - q_{j-1}^j | v_j q_j^{j+1} \dots q_{k-1}^k v_k \\
	&\leq \epsilon (\Csigma (\log n) K^{-1})^{k-j} \sum_{x_i^{j-1} \in [K]^{j-i}} u_i p_i^{i+1} \dots p_{j-2}^{j-1}  u_{j-1} \\
		&\hspace{3cm} \times  \sum_{x_j \in [K]} (u_j + \epsilon e^{-D}) \dots \sum_{x_k \in [K]} (u_k + \epsilon e^{-D})
\end{align*}
and for all $j \in \{i, \dots, k\}$ (with a special case for $j=i$),
\begin{multline*}
\sum_{x \in [K]} \alpha_{i|k}(x) u(y_i | x)
	= \sum_{x_i^k \in [K]^{k-i+1}} u_i p_i^{i+1} \dots p_{j-2}^{j-1} u_{j-1} p_{j-1}^j u_j p_j^{j+1} \dots p_{k-1}^k u_k \\
	\geq (\Csigma (\log n) K)^{-(k-j+1)} \!\!\!\! \sum_{x_i^{j-1} \in [K]^{j-i}} \!\!\!\! u_i p_i^{i+1} \dots p_{j-2}^{j-1} u_{j-1} \sum_{x_j \in [K]} u_j \dots \sum_{x_k \in [K]} u_k.
\end{multline*}
so that
\begin{align*}
&\frac{
	\sum_{x_i^k \in [K]^{k-i+1}} u_i p_i^{i+1} \dots u_{j-1} | p_{j-1}^j - q_{j-1}^j | v_j \dots q_{k-1}^k v_k
	}{
		\sum_{x_i^k \in [K]^{k-i+1}} u_i p_i^{i+1} \dots u_{j-1} p_{j-1}^j u_j \dots p_{k-1}^k u_k
	} \\
	&\hspace{3cm} \leq
		\epsilon K (\Csigma \log n)^{2(k-j)+1} \prod_{\ell = j}^k \frac{K\epsilon e^{-D} + \sum_{x_\ell} u_\ell}{\sum_{x_\ell} u_\ell} \\
	&\hspace{3cm} \leq
		\epsilon K (\Csigma \log n)^{2(k-j)+1} \prod_{\ell = j}^k \left(1 + \frac{K\epsilon e^{-D}}{Ke^{-D}(1-\epsilon)} \right) \\
	&\hspace{3cm} \leq \epsilon K \frac{(\Csigma \log n)^{2(k-j)+1}}{(1 - \epsilon)^{k-j+1}}.
\end{align*}
Likewise, for all $j \in \{i, \dots, k\}$,
\begin{equation*}
\frac{
	\sum_{x_i^k \in [K]^{k-i+1}} u_i p_i^{i+1} \dots p_{j-1}^j | u_j - v_j | q_j^{j+1} \dots q_{k-1}^k v_k
	}{
		\sum_{x_i^k \in [K]^{k-i+1}} u_i p_i^{i+1} \dots u_{j-1} p_{j-1}^j u_j \dots p_{k-1}^k u_k
	}
	\leq \epsilon K \frac{(\Csigma \log n)^{2(k-j)+1}}{(1 - \epsilon)^{k-j+1}}.
\end{equation*}
Therefore, when $\epsilon \leq 1/2$,
\begin{align*}
\frac{\sum_{x \in [K]} |\alpha_{i|k}(x) u(y_i | x) - \delta_{i|k}(x) v(y_i | x)|}{\sum_{x \in [K]} \alpha_{i|k}(x) u(y_i | x)}
	\leq{} & 2 \frac{\epsilon K}{\Csigma \log n} \sum_{j=i}^k \left( 2 (\Csigma \log n)^2 \right)^{k-j+1} \\
	\leq{} & 4 \epsilon K \Csigma (\log n) \frac{\left( 2 (\Csigma \log n)^2 \right)^{k-i} - 1}{2 (\Csigma \log n)^2 - 1} \\
	\leq{} & 4 \epsilon K \left( \sqrt{2} \Csigma \log n \right)^{2(k-i)+1}
\end{align*}
since $n \geq e^2$, which gives the desired result.
\end{proof}

\begin{lemma}
\label{lemma_taille_crochet_mu}
Assume $\epsilon \leq \frac{1}{2}$, then
\begin{equation*}
\| \nu - \omega \|_1
	\leq 6 \left( \sqrt{2} \Csigma \log n \right)^{2k+3} K \epsilon
\end{equation*}
and
\begin{equation}
\label{eq_taille_crochet_fi}
\sup_{0 \leq i \leq k}  \sup_{x \in [K]} \| f_{i|k}(x, \cdot) - g_{i|k}(x, \cdot) \|_1
	\leq 6 \left( \sqrt{2} \Csigma \log n \right)^{2k+3} K \epsilon.
\end{equation}
\end{lemma}
\begin{proof}
With minimalist notations,
\begin{align*}
\sum |\nu - \omega|
	&= \sum \left| \frac{a \alpha u}{\sum b \delta v} - \frac{b \delta v}{\sum a \alpha u} \right| \\
	&\leq \frac{\sum | a \alpha u - b \delta v |}{\sum b \delta v}
		+ \sum |b \delta v| \left| \frac{1}{\sum a \alpha u} - \frac{1}{\sum b \delta v} \right| \\
	&\leq \frac{\sum | a \alpha u - b \delta v |}{\sum b \delta v}
		+ \frac{\sum | a \alpha u - b \delta v |}{\sum a \alpha u} \\
	&\leq 2 \Csigma (\log n) K \frac{\sum | a \alpha u - b \delta v |}{\sum \alpha u}
\end{align*}
using $(\Csigma (\log n) K)^{-1} \leq a \leq b \leq \Csigma (\log n) K^{-1}$, $0 \leq \alpha \leq \delta$ and $0 \leq u \leq v$. Thus,
\begin{align*}
\sum |\nu - \omega|
	&\leq 2 \Csigma (\log n) K \left(
			\frac{\sum b | \alpha u - \delta v |}{\sum \alpha u}
			+ \frac{\sum | a - b | \alpha u}{\sum \alpha u}
		\right)	\\
	&\leq 2 \Csigma (\log n) K \left( 
			\Csigma (\log n) K^{-1} \frac{\sum | \alpha u - \delta v |}{\sum \alpha u}
			+ \epsilon
		\right) \\
	&\leq 2 \Csigma (\log n) K \left( 
			\Csigma (\log n) 4 \left( \sqrt{2} \Csigma \log n \right)^{2k+1} \epsilon
			+ \epsilon
		\right) \  \text{(Lemma~\ref{lemma_taille_crochet_beta})} \\
	&\leq 6 \left( \sqrt{2} \Csigma \log n \right)^{2k+3} K \epsilon.
\end{align*}
The control of $\sum_{x' \in [K]} |g_{i|k} - f_{i|k}|(x,x')$ is the same after replacing $a$ and $b$ by $p$ and $q$.
\end{proof}

Write $\eta = 6 \left( \sqrt{2} \Csigma \log n \right)^{2k+3} K \epsilon$.
Equation (\ref{eq_taille_crochet_fi}) implies that as soon as $\eta < 1$, it is possible to enlarge the bracket $[f_{i|k}, g_{i|k}]$ into a bracket $[f'_{i|k}, g'_{i|k}]$ of size smaller than $3 \eta$ for the norm of Lemma~\ref{lemma_taille_crochet_mu} such that $f'_{i|k} / (1 - \eta)$ and $g'_{i|k} / (1 + \eta)$ are transition matrices.


Let
\begin{equation*}
\begin{cases}
\displaystyle A_x(y_0^k) = \left( \nu f'_{1|k-1} \dots f'_{k-1|k-1} p \right)_x u(y_k | x) \\ \\
\displaystyle B_x(y_0^k) = \left( \omega g'_{1|k-1} \dots g'_{k-1|k-1} q \right)_x v(y_k | x)
\end{cases}
.
\end{equation*}
$[A,B]$ is a bracket of $\Gcal$, and this construction gives a bracket covering of $\Gcal$.

\begin{lemma}
\label{lemma_taille_crochet_justeavantA}
Assume $\epsilon \leq \frac{1}{12 k \left( \sqrt{2} \Csigma \log n \right)^{2k+3} K}$. Then for all $y_0^k$,
\begin{equation*}
\sum_{x \in [K]} | (\nu f'_{1|k} \dots f'_{k|k})_x - (\omega g'_{1|k} \dots g'_{k|k})_x |
	\leq 7 k \eta = 42 k \left( \sqrt{2} \Csigma \log n \right)^{2k+3} K \epsilon
\end{equation*}
and
\begin{equation*}
\sum_{x \in [K]} | (\nu f'_{1|k} \dots f'_{k|k} p)_x - (\omega g'_{1|k} \dots g'_{k|k} q)_x |
	\leq 64 k \left( \sqrt{2} \Csigma \log n \right)^{2k+3} K \epsilon.
\end{equation*}
\end{lemma}
\begin{proof}
First,
\begin{multline*}
\sum_{x \in [K]} | (\nu f'_{1|k} \dots f'_{k|k})_x - (\omega g'_{1|k} \dots g'_{k|k})_x |
	\leq \sum_{x \in [K]} | ((\nu - \omega) f'_{1|k} \dots f'_{k|k})_x | \\
		+ \sum_{j=1}^k \sum_{x \in [K]} | (\omega g'_{1|k} \dots g'_{j-1|k} (g'_{j|k} - f'_{j|k}) f'_{j+1|k} \dots f'_{k|k})_x |.
\end{multline*}
Then, since $f'_{i|k} / (1 - \eta)$ and $g'_{i|k} / (1 + \eta)$ are transition matrices (and thus are 1-Lipschitz linear operators of $\Lbf^1([K])$),
\begin{multline*}
\| \nu f'_{1|k} \dots f'_{k|k} - \omega g'_{1|k} \dots g'_{k|k} \|_1
	\leq \| \omega - \nu \|_1 (1-\eta)^k \\
		+ \sum_{j=1}^k \| \omega \|_1 (1 + \eta)^{j-1} \left( \sup_{1 \leq i \leq k} \sup_{x \in [K]} \| f'_{i|k}(x, \cdot) - g'_{i|k}(x, \cdot) \|_1 \right) (1 - \eta)^{k-j}.
\end{multline*}
By Lemma~\ref{lemma_taille_crochet_mu}, $\| \omega \|_1 \leq 1 + \eta$ (since the bracket $[\nu, \omega]$ contains a probability distribution $\mu^\theta_{0|k}$ for some $\theta \in S_{K,M}$) and $\sup_{1 \leq i \leq k} \sup_{x \in [K]} \| f'_{i|k}(x, \cdot) - g'_{i|k}(x, \cdot) \|_1 \leq 3\eta$, so that
\begin{align*}
\| \nu f'_{1|k} \dots f'_{k|k} - \omega g'_{1|k} \dots g'_{k|k} \|_1
	\leq{} & \eta
		+ (1 + \eta) \sum_{j=1}^k (1 + \eta)^{j-1} 3 \eta \\
	\leq{} & \eta \left( 1 + 3 (1 + \eta) \sum_{j=0}^{k-1} (1 + \eta)^{j} \right) \\
	\leq{} & \eta \left( 1 + 3 (1 + \eta) \frac{(1 + \eta)^k - 1}{\eta} \right) \\
	\leq{} & \eta + 3 (1+\eta) (e^{k \eta} - 1).
\end{align*}
For all $x \in [0, \frac{1}{2}]$, $3 (1+x) (e^{x} - 1) \leq 6x$. Since $k\eta \leq \frac{1}{2}$ by the assumption on $\epsilon$,
\begin{align*}
\| \nu f'_{1|k} \dots f'_{k|k} - \omega g'_{1|k} \dots g'_{k|k} \|_1
	\leq{} & \eta + 6 k \eta \leq 7 k \eta.
\end{align*}

For the second part, note that
\begin{align*}
\sum_{x \in [K]} | (\nu f'_{1|k}   &   \dots f'_{k|k} p)_x - (\omega g'_{1|k} \dots g'_{k|k} q)_x | \\
	\leq{} & \sum_x \sum_{x'} | (\nu f'_{1|k} \dots f'_{k|k})_{x'} p_{x',x} - (\omega g'_{1|k} \dots g'_{k|k})_{x'} q_{x',x} | \\
	\leq{} & \sum_x \sum_{x'} | (\nu f'_{1|k} \dots f'_{k|k})_{x'} - (\omega g'_{1|k} \dots g'_{k|k})_{x'} | q_{x',x} \\
		&+ \sum_x \sum_{x'} (\nu f'_{1|k} \dots f'_{k|k})_{x'} | p_{x',x} - q_{x',x} |.
\end{align*}
Since $[p,q]$ is a non-empty bracket of $\{\Qbf_\theta\}_{\theta \in S_{K,M}}$, $\sum_x q_{x',x} \leq 1 + K\epsilon$ for all $x'$ and since $\nu f'_{1|k} \dots f'_{k|k}$ is the lower bound of a non empty bracket of $\{ p_{X_k | Y_1^k, \theta}\}_{\theta \in S_{K,M}}$, ${\sum_{x'} (\nu f'_{1|k} \dots f'_{k|k})_{x'} \leq 1}$. Hence,
\begin{align*}
&\sum_{x \in [K]} | (\nu f'_{1|k} \dots f'_{k|k} p)_x - (\omega g'_{1|k} \dots g'_{k|k} q)_x | \\
	&\leq (1 + K \epsilon) \sum_{x'} | (\nu f'_{1|k} \dots f'_{k|k})_{x'} - (\omega g'_{1|k} \dots g'_{k|k})_{x'} |
		+ K \epsilon \sum_{x'} (\nu f'_{1|k} \dots f'_{k|k})_{x'} \\
	&\leq (1 + K \epsilon) 42 \left( \sqrt{2} \Csigma \log n \right)^{2k+3} K \epsilon + K \epsilon \qquad \text{(by the first part of the lemma)} \\
	&\leq 64 k \left( \sqrt{2} \Csigma \log n \right)^{2k+3} K\epsilon
\end{align*}
since $\epsilon \leq \frac{1}{2K}$ under the assumption of the lemma and $k \wedge (\sqrt{2} \Csigma \log n) \geq 1$.
\end{proof}

\begin{lemma}
\label{lemma_tailleAB}
Assume $\epsilon \leq \frac{1}{12 k \left( \sqrt{2} \Csigma \log n \right)^{2k+1} K}$. Then
\begin{equation*}
d_\Gcal ( A, B ) 
	\leq 70 k \left( \sqrt{2} \Csigma \log n \right)^{2k+1} K \epsilon.
\end{equation*}
\end{lemma}
\begin{proof}
By definition,
\begin{equation*}
d_\Gcal(A,B)
	= \Ebb^*_{Y_0^{k-1}} \sum_{x \in [K]} \int | A_x(Y_0^k) - B_x(Y_0^k)| \lambda(dY_k).
\end{equation*}
Taking some fixed $Y_0^{k-1}$,
\begin{align*}
&\sum_x \int |A_x(y_k) - B_x(y_k)| \lambda(dy_k) \\
	&= \sum_x \int |u(y_k | x) (\nu f'_{1|k-1} \dots f'_{k-1|k-1} p)_x \\
		&\qquad \qquad- v(y_k | x) (\omega g'_{1|k-1} \dots g'_{k-1|k-1} q)_x | \lambda(dy_k) \\
	&\leq \sum_x \int |u(y_k | x) - v(y_k | x)| (\nu f'_{1|k-1} \dots f'_{k-1|k-1} p)_x \lambda(dy_k) \\
		&\quad + \sum_x \int v(y_k | x) | (\nu f'_{1|k-1} \dots f'_{k-1|k-1} p)_x - (\omega g'_{1|k-1} \dots g'_{k-1|k-1} q)_x | \lambda(dy_k).
\end{align*}

Since the brackets are not empty, for all $x \in [K]$, $\int v(y|x) \lambda(dy) \leq 1 + \epsilon e^{-D}$ and $\sum_x (\nu f'_{1|k-1} \dots f'_{k-1|k-1} p)_x \leq 1$ (it is the lower bound of a non empty bracket of $\{ p_{X_k | Y_0^{k-1}, \theta} \, | \, \theta \in S_{K,M} \}$). Therefore, Lemma~\ref{lemma_taille_crochet_justeavantA} entails
\begin{align*}
d_\Gcal(A,B)
	&\leq{} \epsilon e^{-D} \sum_x (\nu f'_{1|k-1} \dots f'_{k-1|k-1} p)_x \\
		& \  + (1 + \epsilon e^{-D}) \sum_x | (\nu f'_{1|k-1} \dots f'_{k-1|k-1} p)_{x} - (\omega g'_{1|k-1} \dots g'_{k-1|k-1} q)_{x} | \\
	&\leq \epsilon e^{-D}
		+ (1 + \epsilon e^{-D}) 64 (k-1) \left( \sqrt{2} \Csigma \log n \right)^{2(k-1)+3} K \epsilon \\
	&\leq 70 k \left( \sqrt{2} \Csigma \log n \right)^{2k+1} K \epsilon
\end{align*}
since $1 + \epsilon e^{-D} \leq 13/12$ under the assumption of the lemma.
\end{proof}

Assume $k \geq 2$ and let $\eta' := 42 (k-1) \left( \sqrt{2} \Csigma \log n \right)^{2k+1} K \epsilon$. Lemma~\ref{lemma_taille_crochet_justeavantA} implies $\sum_x | (\nu f'_{1|k-1} \dots f'_{k-1|k-1})_x - (\omega g'_{1|k-1} \dots g'_{k-1|k-1})_x | \leq \eta'$. Since the bracket $[\nu f'_{1|k-1} \dots f'_{k-1|k-1},$ $\omega g'_{1|k-1} \dots g'_{k-1|k-1}]$ is not empty, it contains a probability measure. Thus, using $(\Csigma \log n)^{-1} K^{-1} \leq p \leq q \leq \Csigma (\log n) K^{-1}$, for all $x \in [K]$,
\begin{align*}
(\Csigma \log n)^{-1} K^{-1} (1 - \eta')
	&\leq (\nu f'_{1|k-1} \dots f'_{k-1|k-1} p)_x \\
	&\leq (\omega g'_{1|k-1} \dots g'_{k-1|k-1} q)_x
	\leq \Csigma (\log n) K^{-1} (1 + \eta').
\end{align*}

Therefore, by equation~\eqref{eq_encadrement_bracketuv},
\begin{multline*}
(\Csigma \log n)^{-1} K^{-1} (1 - \eta') e^{-D} K(1 - \epsilon)
	\leq \sum_{x \in [K]} A_x \\
	\leq \sum_{x \in [K]} B_x
	\leq \Csigma (\log n) K^{-1} (1 + \eta') K ( e^D + \epsilon e^{-D} ).
\end{multline*}

The inequality $(2 \Csigma (\log n) e^D)^{-1} \! \leq \sum_{x \in [K]} \! A_x \leq \sum_{x \in [K]} \! B_x \leq 2 \Csigma (\log n) e^D$ required in the definition of $\bar{N}$ follows as soon as $(1-\eta') (1-\epsilon) \geq 1/2$ and $(1 + \eta') (1 + \epsilon e^{-2D}) \leq 2$, for instance when $(1 - \eta')^2 \geq 1/2$ since $\eta' \geq \epsilon$ and $D \geq 0$, which holds when $\eta' \leq 1/4$, in other words when
\begin{align*}
\epsilon \leq \frac{1}{168 (k-1) \left( \sqrt{2} \Csigma \log n \right)^{2k+1} K}.
\end{align*}

Thus, taking $\epsilon' = 70 k \left( \sqrt{2} \Csigma \log n \right)^{2k+1} K \epsilon$ ensures that if $\epsilon' \leq \frac{70}{168}$, then $d_\Gcal(A,B) \leq \epsilon'$. Lemma~\ref{lemma_lien_entropie_Gcal_et_ens_simples} follows.

\subsubsection{Control of the bracketing entropy of the simple sets and synthesis}

\begin{lemma}
Let $\delta > 0$, then
\begin{equation*}
N\left(\{ \pi_\theta \}_{\theta \in S_{K,M}}, d_\infty, \delta \right)
	\leq \max \left( \frac{K-1}{\delta} , 1 \right)^{K-1},
\end{equation*}
\begin{equation*}
N\left(\{ \Qbf_\theta \}_{\theta \in S_{K,M}}, d_\infty, \delta \right)
	\leq \max \left( \frac{K-1}{\delta} , 1\right)^{K(K-1)},
\end{equation*}
\end{lemma}

Let $\Caux' = \Caux e^{D} \vee (K-1)$, then by \ref{Aentropy},
\begin{equation*}
N\left(\{ \gamma_\theta \}_{\theta \in S_{K,M}}, d_\infty, \delta e^{-D} \right)
	\leq \max \left( \frac{\Caux'}{\delta} , 1\right)^{ m_M K}.
\end{equation*}

Then, Lemma~\ref{lemma_lien_entropie_Gcal_et_ens_simples} ensures that for all $\epsilon \leq \frac{70}{168}$,
\begin{equation*}
\log \bar{N} \left(\Gcal, d_\Gcal, \epsilon \right)
	\leq ( m_M K + K^2 - 1) \log \max \left( \frac{70 k \left( \sqrt{2} \Csigma \log n \right)^{2k+1} K \Caux'}{\epsilon}, 1\right),
\end{equation*}
so that using Equation~\eqref{eq_lien_entropie_Bsigma_et_G} and letting $H(u) = H(\{t^{(D)}_\theta \; | \; \theta \in \Bbf_\sigma\}, d_k, u)$, one gets for all $\epsilon \leq 16 D \Csigma (\log n) e^{2D} \sqrt{70/168}$ and in particular for all $\epsilon \leq 7 D (\sqrt{2} \Csigma \log n) e^{2D}$:
\begin{align*}
H(\epsilon) &\leq ( m_M K + K^2 - 1) \log \max \left( 
		\frac{(16 D \Csigma (\log n) e^{2D})^2  70 k \left( \sqrt{2} \Csigma \log n \right)^{2k+1} K \Caux'}{\epsilon^2}
		, 1 \right) \\
	&\leq 2 ( m_M K + K^2 - 1) \log \max \left( 
		\frac{95 D e^{2D}  \left( \sqrt{2} \Csigma \log n \right)^{k+3/2} \sqrt{k K \Caux'}}{\epsilon}
		, 1 \right).
\end{align*}
Thus, for all $\epsilon > 0$,
\begin{multline*}
H(\epsilon) \leq 2 ( m_M K + K^2 - 1) \log \max \Bigg( 
		\frac{95 D e^{2D}  \left( \sqrt{2} \Csigma \log n \right)^{k+3/2} \sqrt{k K \Caux'}}{\epsilon}, \\
		14 \left( \sqrt{2} \Csigma \log n \right)^{k+1/2} \sqrt{k K \Caux'} \Bigg).
\end{multline*}

\subsection{Choice of parameters}
\label{sec_choice_parameters}

The goal of this section is to find a function $\varphi$ and a constant $C$ for which equation~\eqref{majoration_E} holds, and to choose the weights $x_{K,M}$ of Lemma~\ref{lemma_majorationW}.

\begin{lemma}
\label{lemme.entropie_et_phi}
Let $A, B, C \in \Rbb_+^*$, $H : x \in \Rbb_+^* \mapsto A \log \max(\frac{B}{x}, C)$, and $\varphi(x) : x \in \Rbb_+^* \mapsto x \sqrt{\pi A} (1 + \sqrt{\log \max(\frac{B}{x}, C)})$. Then:
\begin{equation*}
\begin{cases}
\displaystyle x^2 H(x) \leq \varphi(x)^2, \\
\displaystyle \int_0^x \sqrt{H(u)} du \leq \varphi(x).
\end{cases}
\end{equation*}
\end{lemma}

Let
\begin{multline*}
\varphi(u) = u \sqrt{2 \pi ( m_M K + K^2 - 1)} \Bigg( 1 + \\
	\Bigg\{ \log \max \Bigg( 
		\frac{95 D e^{2D}  \left( \sqrt{2} \Csigma \log n \right)^{k+3/2} \sqrt{k K \Caux'}}{u}, \\
		14 \left( \sqrt{2} \Csigma \log n \right)^{k+1/2} \sqrt{k K \Caux'} \Bigg) \Bigg\}^{1/2} \Bigg).
\end{multline*}

The function $x \mapsto \frac{\varphi(x)}{x}$ is nonincreasing, so $x \mapsto \frac{\varphi(x)}{x^2}$ is decreasing and one can define $\sigma_{K,M}$ as the unique solution of the equation $(1 + 2 \sqrt{D} \log n) \varphi(x) = \sqrt{n} x^2$ with unknown $x$, when a solution exists. By the definition of $E$ in Lemma~\ref{lemma_inegalite_concentration},
\begin{align*}
\forall \sigma \geq \sigma_{K,M}, \quad E
	&\leq \sqrt{n} \varphi(\sigma) + 4D (\log n)^2 \frac{\varphi(\sigma)^2}{\sigma^2} \\
	&\leq \left(1 + \frac{4D(\log n)^2}{1 + 2\sqrt{D} \log n} \right) \varphi(\sigma) \sqrt{n} \\
	&\leq \left(1 + 2\sqrt{D} \log n \right) \varphi(\sigma) \sqrt{n}.
\end{align*}

Using equation~\eqref{majoration_Z_choix_x_a_faire}, for all $z > 0$ and $x_{K,M} \geq \sigma_{K,M}$, with probability larger than $1 - e^{-z}$,
\begin{align*}
W_{K,M} \! \leq{} & 4 C^* (n_* \! +k+1) \! \left[
			(1 + 2\sqrt{D} \log n) \frac{\varphi(x_{K,M})}{x_{K,M}^2 \sqrt{n}}
			+ \!\! \sqrt{\frac{z}{x_{K,M}^2 n}}
			+ \! 4D \frac{z (\log n)^2}{x_{K,M}^2 n}
		\right] \\
	\leq{} & 4 C^* (n_* \! +k+1) \! \left[
			\frac{\sigma_{K,M}}{x_{K,M}}
			+ \sqrt{\frac{z}{x_{K,M}^2 n}}
			+ 4D (\log n)^2 \frac{z}{x_{K,M}^2 n}
		\right].
\end{align*}

Let $\epsilon > 0$, and let us take
\begin{equation*}
x_{K,M} = \frac{1}{\theta} \left(\sigma_{K,M} + \sqrt{\frac{z}{n}} \right),
\end{equation*}
where $\theta > 0$ is such that $2 \theta + 4D (\log n)^2 \theta^2 \leq \frac{\epsilon}{4 C^* (n_*+k+1)}$. Then
\begin{align*}
W_{K,M}
	\leq 4 C^* (n_*+k+1) \left[
			\theta
			+ \theta
			+ 4D (\log n)^2 \theta^2
		\right]
	\leq \epsilon
\end{align*}
and
\begin{align*}
W_{K,M} x_{K,M}^2
	\leq{} & 4 C^* (n_*+k+1) \left[
			\sigma_{K,M} x_{K,M} + \sqrt{\frac{z}{n}} x_{K,M} + 4D (\log n)^2 \frac{z}{n}
		\right] \\
	\leq{} & 4 C^* (n_*+k+1) \left[
			\theta x_{K,M}^2 + 4D (\log n)^2 \frac{z}{n}
		\right] \\
	\leq{} & 8 C^* (n_*+k+1) \left[
			\frac{1}{\theta} \sigma_{K,M}^2 + \left(4D (\log n)^2 + \frac{1}{\theta}\right) \frac{t}{n}
		\right].
\end{align*}

Take $z = s + w_M + K$, then since $\sum_M e^{-w_M} \leq e-1$, with probability larger than $1 - e^{-s}$, for all $M$, $K$ and for all functions $\pen$ such that
\begin{equation*}
\pen_n(K,M) \geq 8 C^* (n_*+k+1) \left[
			\frac{1}{\theta} \sigma_{K,M}^2 + \left(4D (\log n)^2 + \frac{1}{\theta}\right) \frac{w_M + K}{n}
		\right],
\end{equation*}
it holds
\begin{equation*}
W_{K,M} x_{K,M}^2 - \pen_n(K,M) \leq 8 C^* (n_*+k+1) \left(4D (\log n)^2 + \frac{1}{\theta}\right) \frac{s}{n}.
\end{equation*}

A $\theta$ that satisfies $2\theta + 4D (\log n)^2 \theta^2 = \frac{\epsilon}{4C^*(n_*+k+1)}$ is
\begin{equation*}
\theta = \frac{1}{4D (\log n)^2} \left( \sqrt{1 + \frac{\epsilon D (\log n)^2}{C^* (n_*+k+1)}} - 1 \right).
\end{equation*}

Let us take this $\theta$. Since $\frac{1}{\sqrt{1+x} - 1} \leq \max(1, \frac{3}{x})$ for all $x > 0$,
\begin{align*}
\frac{1}{\theta} &\leq 12 C^* (n_*+k+1) \max\left( \frac{D (\log n)^2}{3C^* (n_*+k+1)}, \frac{1}{\epsilon} \right).
\end{align*}

Therefore,
\begin{align*}
W_{K,M} & x_{K,M}^2 - \pen_n(K,M) \\
	&\leq 96 (C^*)^2 (n_*+k+1)^2 \left(\frac{D (\log n)^2}{3C^* (n_*+k+1)} + \frac{1}{\epsilon} \vee \frac{D (\log n)^2}{3C^* (n_*+k+1)} \right) \frac{s}{n} \\
	&\leq 192 (C^*)^2 (n_*+k+1)^2 \left(\frac{1}{\epsilon} \vee \frac{D (\log n)^2}{3C^* (n_*+k+1)} \right) \frac{s}{n}
\end{align*}
as soon as
\begin{equation*}
\pen_n(K,M) \geq 96 (C^*)^2 (n_*+k+1)^2 \left( \frac{1}{\epsilon} \vee \frac{D (\log n)^2}{3C^* (n_*+k+1)} \right)  
	\left( \sigma_{K,M}^2 + 2 \frac{w_M + K}{n} \right).
\end{equation*}

The last step of the proof is to find an upper bound of $\sigma_{K,M}$.

\begin{lemma}
Let $A$, $B$, $C$ and $E$ be functions $\Nbb \longrightarrow [1, \infty)$, and $\varphi_n : x \longmapsto x A(n) (1 + \sqrt{\log \max (\frac{B(n)}{x}, C(n))})$. Let $\sigma_n$ be the only solution of the equation $\frac{\varphi_n(x)}{x^2 \sqrt{n}} = \frac{1}{E(n)}$ with unknown $x \in \Rbb_+^*$. Let
\begin{equation*}
f(n) = \left[ \frac{A(n)C(n)E(n)}{B(n)}(1 + \sqrt{\log B(n) + \log n}) \right]^2.
\end{equation*}

Assume that there exists $n_1$ such that for all $n \geq n_1$, $f(n) \leq n$. Then
\begin{equation*}
\forall n \geq n_1, \qquad \sigma_n \leq \frac{A(n)E(n)}{\sqrt{n}}(1 + \sqrt{\log B(n) + \log n}).
\end{equation*}
\end{lemma}

In our case,
\begin{equation*}
\begin{cases}
A(n) = \sqrt{2 \pi ( m_M K + K^2 - 1)} , \\
B(n) = 95 D e^{2D}  ( \sqrt{2} \Csigma \log n )^{k+3/2} \sqrt{k K \Caux'} , \\
C(n) = 14 ( \sqrt{2} \Csigma \log n )^{k+1/2} \sqrt{k K \Caux'} , \\
E(n) = 1 + 2 \sqrt{D} \log n \leq 3 \sqrt{D} \log n .
\end{cases}
\end{equation*}

Hence
\begin{align*}
f(n)
	&\leq 18 \pi \, ( m_M K + K^2 - 1) D (\log n)^2 \left( \frac{14}{95 D e^{2D}} \right)^2 \left(1 + \sqrt{\log B(n) + \log n} \right)^2 \\
	&\leq \frac{4}{5} \pi \, ( m_M K + K^2 - 1) (\log n)^2 \frac{e^{-4D}}{D} \Bigg(1 + \log n +
			\log 95 + \log D + 2D + \\
		& \hspace{4cm} \left(k+\frac{3}{2}\right) \log (\sqrt{2}\Csigma \log n) + \frac{1}{2} \log (k K \Caux')
		\Bigg) \\
	&\leq \frac{4}{5} \pi \, ( m_M K + K^2 - 1) (\log n)^2 \frac{e^{-4D}}{D} \Bigg(15 D + 2k \log \log n + \frac{1}{2} \log \Caux
		\Bigg)
\end{align*}
when $\log n \geq \sqrt{2} \Csigma \geq 1$ by using that $1 \leq k, K \leq n$, $\log x \leq x$ for all $x \geq 0$, $D \geq \log n$ by assumption and $\log \Caux' \leq \log \Caux + D + \log K$. Thus,
\begin{equation*}
f(n) \leq \tilde{f}_{K,M}(n) :=  14 \pi \, ( m_M K + K^2 - 1) e^{-4D} (\log n)^2 (k + \log \Caux).
\end{equation*}

Now, assume that there exists $n_1$ such that $\tilde{f}_{K,M}(n) \leq n$ for all $n \geq n_1$, then for all $n \geq n_1$,
\begin{align*}
\sigma_{K,M}^2
	&\leq \frac{36 \pi \, ( m_M K + K^2 - 1) D (\log n)^2}{n} (1 + \log n + \log B) \\
	&\leq \frac{36 \pi \, ( m_M K + K^2 - 1) D (\log n)^2}{n} \left(15 D + 2 k \log \log n + \frac{1}{2} \log \Caux\right).
\end{align*}

Therefore, there exists a numerical constant $C_\pen$ such that the condition on the penalty is implied by
\begin{multline*}
\pen_n(K,M)
	\geq \frac{C_\pen}{n} (n_*+k+1)^2 \left( \frac{1}{\epsilon} \vee \frac{D (\log n)^2}{3C^* (n_*+k+1)} \right)  
		\Big( w_M + \\
		 ( m_M K + K^2 - 1) D (\log n)^2 (D + k \log \log n + \log \Caux) \Big).
\end{multline*}

\end{document}